\documentclass{article}
\usepackage{psfrag}
\usepackage[dvips]{graphicx}
\usepackage[rflt]{floatflt}
\usepackage{subfigure}
\usepackage{amsmath}
\usepackage{amsthm}
\usepackage{amssymb}
\usepackage{enumerate}
\usepackage{setspace}
\usepackage{epsfig}
\usepackage{fancybox}
\usepackage{float}
\usepackage{times}
\usepackage{array}
\usepackage{makeidx}
\usepackage{verbatim}

\newtheorem*{theorem_nonumber}{Theorem}
\newtheorem{theorem}{Theorem}[section]
\newtheorem{proposition}{Proposition}[section]
\newtheorem{corollary}[theorem]{Corollary}
\newtheorem{lemma}[theorem]{Lemma}
\newtheorem*{claim}{Claim}

\theoremstyle{definition}

\theoremstyle{remark}

\newfont{\valami}{ptmr8r scaled 1200} 
\newfont{\kisvalami}{ptmr8r scaled 1000}

\newcommand{\dto}{\backslash.}
\newcommand{\dof}{\backslash}
\newcommand{\cto}{/.}
\newcommand{\cof}{/}

\begin{document}

\author{
Konstantinos Papalamprou \\
London School of Economics
\and
Leonidas Pitsoulis  \\
Aristotle University of Thessaloniki 
}
\title{Decomposition of Binary Signed-Graphic Matroids}
\maketitle

\begin{abstract}
In this paper we employ Tutte's theory of bridges to derive a decomposition theorem for binary matroids arising from 
signed graphs. The proposed decomposition differs from previous decomposition results on matroids that have 
appeared in the literature in the sense that it is not based on $k$-sums, but rather on the operation of deletion of a cocircuit. 
Specifically, it is shown that certain minors resulting from the deletion of a cocircuit of a binary matroid will be graphic matroids
apart from exactly one that will be signed-graphic, if and only if the matroid is signed-graphic. 
\end{abstract}

\section{Introduction} \label{sec_intro}
The theory of bridges was developed by Tutte in~\cite{Tutte:1959} in order to answer fundamental questions regarding graphs
and their matroids, such as when a binary matroid is graphic. Moreover, in his latest book~\cite{Tutte:98} he expressed the belief that 
this theory is rich enough to  enjoy more theoretical applications. In this work we use the theory of bridges to derive a 
decomposition result for binary signed-graphic matroids. 
The main result is the following theorem 
which states that deletion of a cocircuit naturally decomposes a binary signed-graphic matroid into minors 
which are all graphic apart from one which is signed-graphic, while these conditions are also sufficient for a binary 
matroid to be signed-graphic. 
\begin{theorem_nonumber}[Decomposition]
Let $M$ be a connected binary matroid and $Y\in \mathcal{C}^{*}(M)$ be a non-graphic cocircuit. Then $M$ is signed-graphic if and only if:
\begin{itemize}
 \item [(i)] $Y$ is bridge-separable, and
 \item [(ii)] the $Y$-components of $M$ are all graphic apart from one which is signed-graphic.
\end{itemize}
\end{theorem_nonumber}
\noindent
This decomposition follows the lines of an analogous result for graphic matroids by Tutte 
in~\cite{Tutte:1959, Tutte:1965}, however it differs in many ways mainly due to the more complex nature of cocircuits 
in signed-graphic matroids with respect to cocircuits in graphic matroids.

Signed-graphic matroids have attracted the attention of many researchers over the past 
years (see~\cite{Pagano:1998,Slilaty:2005b,Slilaty:06,SliQin:07,Zaslavsky:1990,Zaslavsky:1991a} among others), while
recently it has also been conjectured that they may be the building blocks of a $k$-sum decomposition of dyadic and near-regular 
matroids~\cite{Whittle:05}. 
An overview of previous decomposition results regarding  signed-graphic matroids and signed graphs can be found in~\cite{SliQin:07}. 
However, the majority of the results presented in that work are mainly decomposition results for signed graphs rather than for 
signed-graphic matroids. Specifically, based on previous results of Pagano~\cite{Pagano:1998} and Gerards~\cite{Gerards:90}, 
the authors of~\cite{SliQin:07} provide two main decomposition theorems for a signed graph $\Sigma$; one theorem concerning 
the case in which the associated signed-graphic matroid $M(\Sigma)$ is binary and one theorem concerning the case
 in which $M(\Sigma)$ is quaternary. The notion of $k$-sums of signed graphs is introduced by Pagano in~\cite{Pagano:1998} 
while Gerards introduces the similar notion of $k$-splits $(k=1,2,3)$ in order to provide decomposition results for signed graphs 
whose complete lift matroids are regular (see~\cite{Zaslavsky:1991a} for a definition of the complete lift matroid of a signed graph). 
In~\cite{SliQin:07}, these notions are slightly altered and extended so that the signed-graphic matroid of the $k$-sum of two signed 
graphs $\Sigma_1$ and $\Sigma_2$ will be equal to the matroidal $k$-sum of the associated signed-graphic matroid $M(\Sigma_1)$ 
and $M(\Sigma_2)$. By using these $k$-sum operations of signed graphs and the results of~\cite{Gerards:90,Pagano:1998}, the 
above mentioned decomposition theorems regarding the class of signed graphs with binary or quaternary matroids are   proved in~\cite{SliQin:07}.

The paper is structured in the following way. Section~\ref{sec_preliminaries} presents all the necessary theory about graphs and 
matroids. Bonds in signed graphs, which play a central role in this work, are classified in this section and the connection
with the cocircuits in the corresponding signed-graphic matroid is made. 
In section~\ref{subsec_tangled_graphs} we restrict ourselves to binary signed-graphic matroids and their graphical representations, 
tangled signed graphs. 
Section~\ref{sec_decomposition} is the main section of this paper, where the necessary structural theorems which provide
the connection between a tangled signed graph and its corresponding matroid are presented. These theorems eventually
lead to the decomposition Theorem~\ref{thrm_decomposition}  at the end of this section.

\section{Preliminaries} \label{sec_preliminaries}

The main references for graphs and signed graphs are~\cite{Diestel:05,Tutte:01} and~\cite{Zaslavsky:1982,Zaslavsky:1991a}
respectively, while for matroid theory is the book of Oxley~\cite{Oxley:92}. In this section we will mention some not so  basic
operations that will be frequently used in the paper.

\subsection{Graphs and Signed Graphs} \label{subsec_signed_graphs}
By a graph $G:=(V,E)$ we mean a finite set of vertices $V$, and a multiset of
edges $E$. 
Given two distinct vertices $v,u\in V$ we  have four types of edges: $e=\{u, v \}$ is called
a \emph{link}, $e=\{v, v \}$ a \emph{loop}, $e=\{v\}$ a \emph{half edge}, 
while $e=\emptyset$ is a \emph{loose edge}. Whenever applicable, the vertices
that define an edge are called its \emph{end-vertices}. We say that an edge 
$e$ is \emph{incident} to a vertex $v$ if $v\in e$. 
Observe that the above is the ordinary definition of a graph, except that we
also allow half edges and loose edges. We will denote the set of vertices and the set of edges
of a graph $G$ by $V(G)$ and $E(G)$, respectively. 
The \emph{deletion of an edge} $e$ from $G$ is the subgraph defined as 
$G\dof e := (V(G), E(G)-e)$. \emph{Identifying}  two vertices $u$ and $v$ is the operation 
where we replace $u$ and $v$ with a new vertex $v'$ in both $V(G)$ and $E(G)$. 
The \emph{contraction of a link} $e=\{u,v\}$ is the subgraph denoted by
$G/e$ which results from $G$ by identifying $u,v$ in $G\dof e$.
The \emph{contraction of a half edge} $e=\{v\}$ or a \emph{loop} $e=\{v\}$ is the subgraph denoted by
$G/e$ which results from the removal of $\{v\}$ and all half edges and loops incident to it, while all other links
incident to $v$ become half edges at their other end vertex. Contraction of a loose edge is the same as deletion. 
The \emph{deletion of a vertex} $v$ of $G$ is defined as the deletion of all edges incident to $v$ and 
the deletion of $v$ from $V(G)$.  
A graph $G'$ is called a \emph{minor} of $G$ if it is obtained from a sequence of deletions and contractions 
of edges and deletions of vertices of $G$. For $S\subseteq E(G)$, we say that the subgraph $H$ of $G$ is the 
\emph{deletion} of $G$ \emph{to} $S$, denoted by $H=G\dto{S}$, if  $E(H)=S$ and $V(H)$ is the set of end-vertices of 
all edges in $S$. Clearly for set $S\subseteq E(G)$, $G\dto{S}$ is the graph obtained from $G\dof{E(G)-S}$ by 
deleting the isolated vertices (if any). Moreover, for $S\subseteq E(G)$, a subgraph $K$ of $G$ is 
the \emph{contraction} of $G$ \emph{to} $S$, denoted by $K=G\cto{S}$, if $K$ is the graph obtained from $G/(E(G)-S)$ 
by deleting the isolated vertices (if any).
Any partition $\{T,U\}$ of $V(G)$ for nonempty $T$ and $U$, defines a \emph{cut} of $G$ denoted by $E(T,U)\subseteq{E(G)}$
as the set of links incident to a vertex in $T$ and a vertex in $U$. A cut of the form
$E(v,V(G)-v)$ is called the \emph{star} of vertex $v$. 
There are several definitions of  connectivity in graphs that have appeared in the literature.  
In this paper we will employ the  \emph{Tutte $k$-connectivity} which we will refer to as $k$-connectivity, due to the fact that the
connectivity of a graph and its corresponding graphic matroid coincide under this definition.  
For $k\geq 1$, a \emph{$k$-separation} of a connected graph $G$ is a partition $\{A,B\}$ of the edges 
such that $\min\{|A|,|B|\}\geq{k}$ and $|V({G\!\!:\!\!A}) \cap V({G\!\!:\!\!B})|=k$, where $G\!\!:\!\!A$ is the subgraph
of $G$ induced by $A$. 
For $k\geq 2$, we say that $G$ is \emph{$k$-connected}  
if $G$ does not have an $l$-separation for $l=1,\ldots,k-1$. 
A \emph{block} is defined as a maximally 2-connected subgraph of $G$. 
Loops and half-edges are always blocks in a graph, since they are 2-connected (actually they are
infinitely connected) and they cannot be part of a 2-connected component because they induce
a 1-separation. 
Finally we define the operation of \emph{reversing}, which is also known  as twisting (see
\cite{Oxley:92}), as follows. Let $G_1$ and
$G_2$ be two disjoint graphs with at least two vertices $(u_1,v_1)$ and
$(u_2,v_2)$, respectively. Let $G$ be the graph obtained from $G_1$ and $G_2$  
by identifying $u_1$ with $u_2$ to a vertex $u\in{V(G)}$ and $v_1$ with $v_2$
to a vertex  $v\in{V(G)}$. If we identify, instead, $u_1$ with $v_2$ and $v_1$ with $u_2$ then we obtain
a graph $G'$ which is called  a \emph{reversed} graph of $G$ \emph{about}
$\{u,v\}$. The subgraphs $G_1$ and $G_2$ of $G$ and $G'$ are called the \emph{reversing parts} of the reversing.

A \emph{signed graph} is defined as $\Sigma := (G,\sigma)$ where $G$ is a graph called the \emph{underlying graph} and
$\sigma$ is a sign function $\sigma:E(G)\rightarrow \{\pm 1\}$, where 
$\sigma(e) =-1$ if $e$ is a half edge and  $\sigma(e) =+1$ if $e$ is a loose edge. 
Therefore a signed graph is a graph where the  edges are labelled as  positive or  negative, while all the half edges
are negative and all the loose edges are positive. 
We denote by $V(\Sigma)$ and $E(\Sigma)$ the vertex set and edge set of a signed graph $\Sigma$, respectively. 
All operations on signed graphs are defined through a corresponding operation on the underlying graph and
the sign function. In the following definitions assume that we have a signed graph $\Sigma=(G,\sigma)$. 
The operation of \emph{switching} at a vertex $v$ results
in a new signed graph $(G,\bar{\sigma})$  where $\bar{\sigma}(e) := - \sigma (e)$ for each link $e$ incident
to $v$, while $\bar{\sigma}(e) := \sigma (e)$ for all other edges. \emph{Deletion} of an edge $e$ is defined as
$\Sigma \dof e := (G\dof e, \sigma)$. The \emph{contraction} of an edge $e$ consists of three cases: 
\begin{enumerate}
\item if $e$ is a half edge, positive loop  or a positive link, then $\Sigma / e := (G/e, \sigma)$. 
\item if $e$ is a negative loop, then $\Sigma / e := (G'/e', \sigma)$ where $G'$ is the graph obtained from
      $G$ by replacing the loop $e$ with a half edge $e'$. 
\item if $e$ is a negative link, then $\Sigma / e := (G/e, \bar{\sigma})$  where $\bar{\sigma}$ is a 
      switching at either one of the end vertices of $e$.
\end{enumerate} 
 The \emph{sign of a cycle} is the product of the signs of its edges, so we have a 
\emph{positive cycle} if the number of negative edges in the cycle is  even, otherwise the cycle is a \emph{negative cycle}.
Both negative loops and  half-edges are negative cycles. A signed graph is 
called \emph{balanced}  if it contains no negative  cycles.
A vertex $v\in V(\Sigma)$ is called a \emph{balancing vertex} if $\Sigma \dof v$ is balanced.

\subsection{Signed-Graphic Matroids} \label{subsec_sign_graph_matrds}
We assume that the reader is familiar with matroid theory as in~\cite{Oxley:92}, and in particular with the
circuit axiomatic definition of a matroid and the notions of duality, connectivity, 
representability and minors. Given a matrix $A$ and a graph $G$, $M[A]$ and
$M(G)$ denote the vector and graphic matroids respectively. For a matroid $M$
we denote by  $E(M)$ be the ground set, $\mathcal{C}(M)$ the family of circuits 
while $M^*$ is the dual matroid of $M$. The prefix `co-' dualizes the term mentioned and the asterisk dualizes the symbol
used. 

The following definition  for the matroid of a signed graph or \emph{signed-graphic matroid} is used in this work. 
\begin{theorem}[Zaslavsky~\cite{Zaslavsky:1982}] \label{th_sgg2}
Given a signed graph $\Sigma$ let $\mathcal{C} \subseteq 2^{E(\Sigma)}$ be the family of edge sets 
inducing a subgraph in $\Sigma$ which is either:
\begin{itemize}
\item[(i)] a positive cycle, or
\item[(ii)]  two vertex-disjoint negative cycles connected by a path which has no common vertex with the cycles apart from 
its end-vertices, or 
\item[(iii)] two negative cycles which have exactly one common vertex.
\end{itemize}
Then $M(\Sigma)=(\mathcal{C}, E(\Sigma))$ is a matroid on $E(\Sigma)$ with circuit family $\mathcal{C}$. 
\end{theorem}
\noindent
The subgraphs of $\Sigma$ induced by the edges corresponding to a circuit of $M(\Sigma)$ are called the 
\emph{circuits} of $\Sigma$. The circuits 
of $\Sigma$ described by (ii) and (iii) of Theorem~\ref{th_sgg2} are also called {\em handcuffs} of Type I and Type II, 
respectively (see Figure~\ref{fig_circuits_signed_graphs}). 
\begin{figure}[hbtp]
\begin{center}
\mbox{
\subfigure[positive cycle]
{
\includegraphics*[scale=0.23]{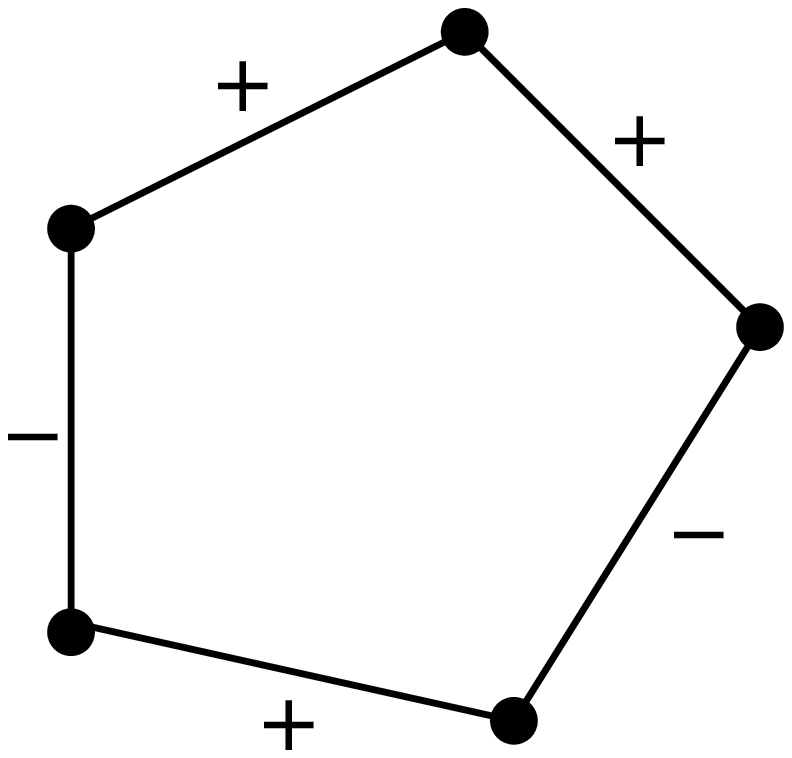}
}\quad
\subfigure[Type I handcuff]
{
\includegraphics*[scale=0.23]{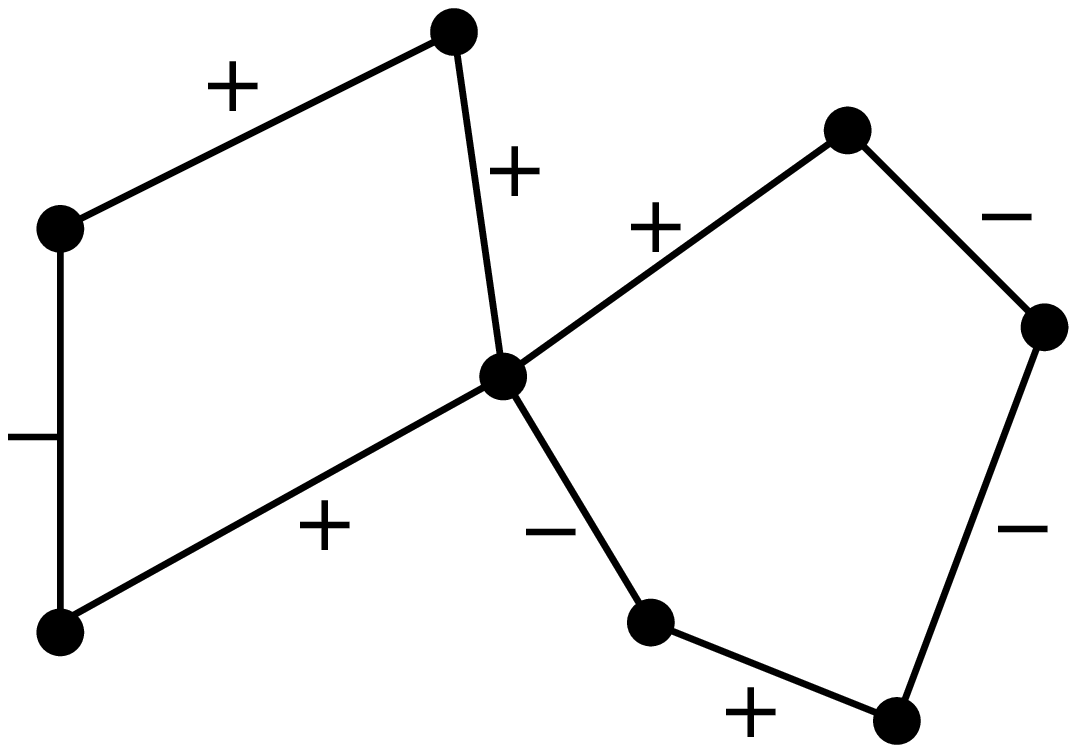}
}\quad
\subfigure[Type II handcuff]
{
\includegraphics*[scale=0.23]{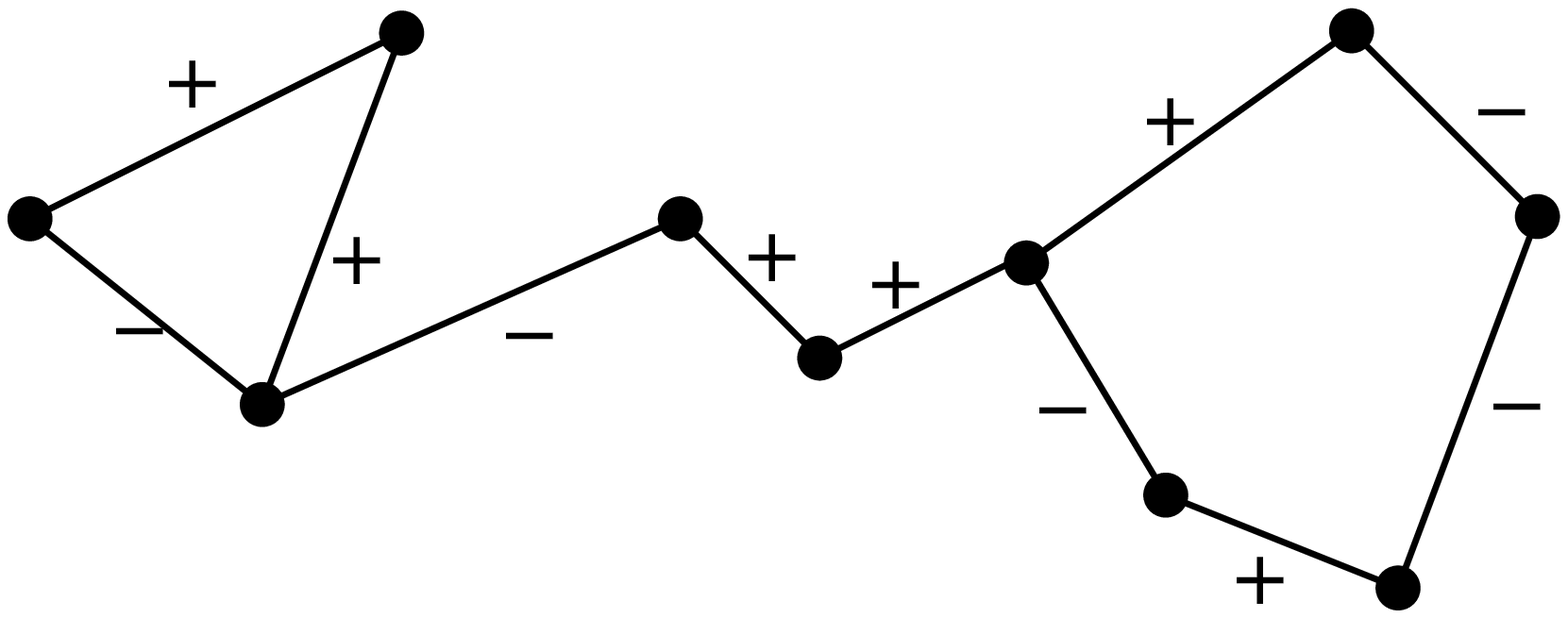}
}}
\end{center}
\caption{Circuits in a signed graph $\Sigma$}
\label{fig_circuits_signed_graphs}
\end{figure}
With the following
theorem we  characterize the sets of edges in a signed graph $\Sigma$ which correspond to circuits
of $M^{*}(\Sigma)$. 
\begin{theorem}[Zaslavsky~\cite{Zaslavsky:1982}] \label{thrm_bonds}
Given a signed graph $\Sigma$ and its corresponding matroid $M(\Sigma)$, $Y\subseteq E(\Sigma)$ is
a cocircuit of $M(\Sigma)$ if and only if $Y$ is a minimal set of edges 
whose deletion increases the number of balanced components of $\Sigma$.
\end{theorem}
\noindent
The sets of edges defined in Theorem~\ref{thrm_bonds} are called \emph{bonds} of 
a signed graph.
In analogy with the different types of circuits a signed-graphic matroid has, 
bonds can also be classified into different types according to the signed
graph obtained upon their deletion. Specifically for a 
given connected and unbalanced signed graph $\Sigma$, the deletion of a bond $Y$ 
results to a signed graph $\Sigma\dof{Y}$  with exactly one
balanced component due to the minimality of $Y$. 
Thus,  $\Sigma\dof{Y}$ may be a balanced connected graph
in which case we call $Y$ a \emph{balancing  bond} or it may consist of one balanced component and some unbalanced
components. In the latter case, if the balanced component is a vertex,
i.e. the balanced component is empty of edges, then we say that
$Y$ is a \emph{star bond}, while in the case that the balanced component is not empty of edges  $Y$
can be either an \emph{unbalanced bond} or a \emph{double bond}.
Specifically, if the balanced component is not empty of edges and there is no edge in $Y$ such that both of its end-vertices are vertices of the balanced component, then $Y$ is an unbalanced bond. On the other hand, if there exists at least one edge of $Y$ whose both end-vertices are vertices of the balanced component then $Y$ is a double bond.
A further classification of bonds is based on whether the  matroid
$M(\Sigma)\dof{Y}$ is connected or not for some $Y\in\mathcal{C}(M^{*}(\Sigma))$. 
In the case that $M(\Sigma)\dof{Y}$ is disconnected we call $Y$ as \emph{separating bond}
of $\Sigma$, otherwise we say that $Y$ is a \emph{nonseparating bond}. 

In \cite{Zaslavsky:1982,Zaslavsky:1991a}
the edge sets of a signed graph which correspond to elementary separators in the associated signed-graphic
matroid are determined. 
Before we present this result in Theorem~\ref{th_Zasl11} we have to provide some necessary definitions.    
An \emph{inner block}  of $\Sigma$ is a block that 
is unbalanced or lies on the path between two unbalanced 
blocks. Any other block is called \emph{outer}. 
The \emph{core} of $\Sigma$  is the union of all inner blocks. 
A \emph{B-necklace} is a special type of 2-connected unbalanced signed graph, 
which is composed of maximally 2-connected balanced 
subgraphs $\Sigma_i$ joined in a cyclic fashion as illustrated in Figure~\ref{fig_necklace}. Note that in 
Figure~\ref{fig_necklace} as well as other figures that follow, a circle depicts a connected graph while two homocentric circles depict a block, 
where in each case a positive (negative) sign is used to indicate whether the connected or 2-connected
component is balanced (unbalanced). 
Observe that any
negative cycle in a B-necklace has to contain at least one edge from each $\Sigma_i$. 
\begin{figure}[h] 
\begin{center}
\centering
\psfrag{S1}{\footnotesize $\Sigma_1$}
\psfrag{S2}{\footnotesize $\Sigma_2$}
\psfrag{S3}{\footnotesize $\Sigma_3$}
\psfrag{S4}{\footnotesize $\Sigma_4$}
\psfrag{S5}{\footnotesize $\Sigma_5$}
\psfrag{Si}{\footnotesize $\Sigma_i$}
\psfrag{Sn}{\footnotesize $\Sigma_n$}
\includegraphics*[scale=0.3]{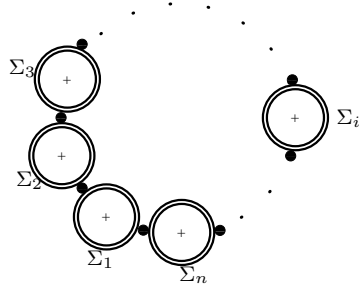}
\end{center}
\caption{A B-necklace}
\label{fig_necklace}
\end{figure}
\begin{theorem}[Zaslavsky~\cite{Zaslavsky:1991a}] \label{th_Zasl11}
Let $\Sigma$ be a connected signed graph. The elementary separators of 
$M(\Sigma)$ are the edge sets of each outer block and the core,  
except that when the core is a B-necklace  each block in the
B-necklace is also an elementary separator. 
\end{theorem}
\noindent
Let $B$ be an elementary separator of $M(\Sigma)$. The subgraph $\Sigma\dto{B}$ of $\Sigma$ is called a \emph{separate} of $\Sigma$.

Given a matroid $M$ and some set $X\subseteq E(M)$ the \emph{deletion} and \emph{contraction} of $X$ from $M$ will be denoted
by $M\dof X$ and $M/X$ respectively.  If $N$ is a minor of $M$, that is  $N=M\dof X / Y$ for disjoint $X,Y\subseteq E(M)$, we will write $M \succeq N$. 
For a matter of convenience in the analysis that will follow
we also employ the complement notions of deletion and contraction, 
that is the \emph{deletion to} a set $X\subseteq E(M)$ is defined as 
\[
M\dto X := M\dof (E(M) - X), 
\]
while the \emph{contraction to} a set $X\subseteq E(M)$ is defined as 
\[
M\cto X := M/(E(M) - X).
\]
There is an equivalence of the aforementioned matroid operations with respect to the associated 
signed-graphic operations of deletion and contraction defined in Section~\ref{subsec_signed_graphs}, 
as indicated by Theorem~\ref{th_minsig}. 
\begin{theorem}[Zaslavsky~\cite{Zaslavsky:1982}] \label{th_minsig}
Let $\Sigma$ be a signed graph and $S\subseteq E(\Sigma)$. Then 
$M(\Sigma\dof{S})=M(\Sigma)\dof{S}$ and $M(\Sigma/S)=M(\Sigma)/S$.
\end{theorem}

The following two propositions provide necessary conditions under which certain operations on a signed graph
do not alter its matroid, and under which a signed-graphic matroid is graphic. Proofs  can be found in, or easily derived from the 
results in~\cite{SliQin:07,Zaslavsky:1982, Zaslavsky:1991a}.
\begin{proposition} \label{prop_samematroid}
Let $\Sigma$ be a signed graph. If $\Sigma'$: 
\begin{itemize}
\item[(i)] is obtained from $\Sigma$ by replacing any number of negative loops by half-edges and vice versa, or
\item[(ii)] is obtained from $\Sigma$ by switchings, or
\item[(iii)] is the reversed graph of $\Sigma$ about $(u,v)$ with $\Sigma_1, \Sigma_2$  the 
reversing parts of $\Sigma$,  where $\Sigma_1$ (or $\Sigma_2$) is balanced or all of its negative cycles contain $u$ and $v$, 
\end{itemize}
then $M(\Sigma) = M(\Sigma')$. 
\end{proposition}
\begin{proposition} \label{prop_sgn_graphic}
Let $\Sigma$ be a signed graph. If $\Sigma$: 
\begin{itemize}
\item[(i)] consists of only positive edges, or
\item[(ii)] is balanced, or 
\item[(iii)] has no negative cycles other than negative loops and half-edges, or
\item[(iv)] has a balancing vertex,
\end{itemize}
then $M(\Sigma)$ is graphic. 
\end{proposition}
\noindent
In the first two cases of Proposition~\ref{prop_sgn_graphic} we also have $M(\Sigma) = M(G)$. For the
third case, there exists a graph $G'$ obtained from $G$ by adding a new vertex $v$ and replacing any
negative loop or half-edge by a link joining its end-vertex with $v$ such that $M(\Sigma) = M(G')$. 
Also a straightforward result which is a  direct consequence of  Proposition~\ref{prop_sgn_graphic} is 
that if $\Sigma$ is a B-necklace then $M(\Sigma)$ is graphic. 
The deletion of any vertex which is common to two balanced blocks in a B-necklace, results in the
elimination of all negative cycles, thereby all such vertices are balancing vertices.

\subsection{Tangled Signed  Graphs} \label{subsec_tangled_graphs}

A connected signed graph is called \emph{tangled}  if it has no balancing vertex and no two vertex disjoint negative cycles. 
For our purposes, the importance of tangled signed graphs stems mainly from
Theorem~\ref{th_sliii} according to which if a binary matroid
is  signed-graphic but not graphic then it  has a tangled graphical
representation. In this section we will provide some preliminary results regarding tangled signed graphs and their matroids. 
\begin{theorem}[Slilaty~\cite{Slilaty:07}]  \label{thrm_tangled1}
If $\Sigma$ is a tangled signed graph then: 
\begin{itemize}
\item[(i)] it contains exactly one unbalanced block, and
\item[(ii)] it does not have a double bond. 
\end{itemize}
\end{theorem}
\noindent
Therefore, if $Y$ is a bond of a tangled signed graph $\Sigma$ then $Y$ is either a star-bond, a balancing bond or an unbalancing bond. 
Clearly if $Y$ is a balancing bond then, provided that $\Sigma$ is connected, the graph $\Sigma\dof{Y}$ consists of one component. 
The next theorem whose proof is omitted shows that if $Y$ is not a balancing bond then $\Sigma\dof{Y}$ consists of exactly two components.
\begin{theorem} \label{thrm_tangled2}
If $\Sigma$ is a tangled signed graph and $Y$ is a star bond or an unbalancing bond, then $\Sigma \dof Y$ consists of
exactly two components and has exactly one unbalancing block. 
\end{theorem}
We work mainly with connected matroids, therefore it would be
desirable to have a connection between the connectivity of a signed-graphic 
matroid $M(\Sigma)$ and the connectivity of $\Sigma$. 
\begin{theorem} \label{th_nero1}
Let $\Sigma$ be a tangled signed graph. Then $\Sigma$ is $2$-connected if and only if $M(\Sigma)$ is connected.
\end{theorem}
\begin{proof}
For the  ``only if'' part, assume that for a $2$-connected tangled signed graph $\Sigma$ the 
matroid $M(\Sigma)$ is disconnected. By Theorem~\ref{th_Zasl11}, this is possible only if $\Sigma$ is a B-necklace. But then
$\Sigma$ contains a balancing vertex and thus, $\Sigma$ is not tangled which is
in contradiction with our assumption.

For the ``if'' part suppose that $M(\Sigma)$ is $2$-connected and it does have a tangled representation $\Sigma$ which 
is not $2$-connected. Therefore $\Sigma$ contains  at least two blocks, and by Theorem~\ref{thrm_tangled1}  exactly one
is unbalancing. By Theorem~\ref{th_Zasl11} $\Sigma$ has two separates, which implies that $M(\Sigma)$ has more than one elementary separators
contradicting our hypothesis about the connectivity of the matroid. 
\end{proof}
The following theorem  can be deduced from $\cite{Pagano:1998,SliQin:07}$. 
\begin{theorem} \label{th_sliii}
If $\Sigma$ is a connected signed graph then $M(\Sigma)$ is binary  if and only if
\begin{itemize}
\item[(i)] $\Sigma$ is tangled, or
\item[(ii)] $M(\Sigma)$ is graphic.
\end{itemize}
\end{theorem}

\section{Decomposition}\label{sec_decomposition}
In this section we will present a decomposition for binary signed-graphic matroids which utilizes the
theory of bridges by Tutte~\cite{Tutte:1959, Tutte:1965}. 
In section~\ref{subsec_bridges} we present some definitions and preliminary results regarding the theory of bridges, which will be needed
for the sections that follow. In section~\ref{subsec_corc_bonds} the cocircuits of binary signed-graphic matroids are further classified
into graphic and non-graphic, depending on whether or not their deletion produces a graphic matroid or not.
An excluded minor characterization  for signed-graphic matroids with all graphic cocircuits is given in section~\ref{subsubsec_graphic}, 
while the decomposition based on non-graphic cocircuits is presented in section~\ref{subsubsec_nongraphic}. The majority of  the results
in this section have to do with the structure of tangled signed graphs, and the relationship between cocircuits in a 
binary signed-graphic matroid and bonds in the corresponding signed graphic representation .

\subsection{Bridges} \label{subsec_bridges}   
Let $Y$ be a cocircuit of a binary matroid $M$. 
We define the \emph{bridges} of $Y$ in $M$ to be the elementary separators of $M\dof{Y}$. If $M\dof{Y}$ has more than one bridge 
then we say that $Y$ is a \emph{separating} cocircuit; otherwise it is \emph{non-separating}. Let $B$ be a bridge of $Y$ in $M$; the
 matroid $M \cto (B\cup{Y})$ is called a \emph{$Y$-component} of $M$. By a result of \cite{Tutte:1959} we know that if $M$ is
 connected then each $Y$-component of $M$ is connected. Furthermore, for any bridge $B$ of $Y$ in $M$, we denote by $\pi(M,B,Y)$ 
the family of all minimal non-null subsets of $Y$ which are intersections of cocircuits of $M\cto({B\cup{Y}})$.  
The following theorem and its corollary relate $\pi(M,B,Y)$ for binary matroids with the family of cocircuits of a given minor.
\begin{theorem}\label{th_ila}
Let $Y$ be a cocircuit of a matroid $M$.
Two elements $a$ and $b$ of $Y$ belong to the same members of $\pi(M,B,Y)$ if and only if they belong to the same cocircuits 
of $M\cto{(B\cup{Y})}\dto{Y}$. 
\end{theorem}
\begin{proof}
For the ``only if'' part, suppose that $a,b\in{W}\in{\pi(M,B,Y)}$. Then for any cocircuit $X$ of $M\cto{(B\cup{Y})}$ either 
$X\cap{W}=\emptyset$ or $W \subseteq{X}$, since otherwise $W$ will not be minimal. This implies that $a$ and $b$ belong to exactly 
the same cocircuits of $M\cto{(B\cup{Y})}$. Thus, by the definition of the matroid operations of contraction and deletion we have 
that $a$ and $b$ belong to the same cocircuits of $M\cto{(B\cup{Y})}\dto{Y}$.

For the ``if'' part, since $a$ and $b$ belong to the same cocircuits of $M\cto{(B\cup{Y})}\dto{Y}$ then by the definition of 
matroid contraction and deletion we obtain that there is no cocircuit $Z$ of $M\cto{(B\cup{Y})}$  such that $Z\cap{\{a,b\}}=\{a\}$ 
or $Z\cap{\{a,b\}}=\{b\}$. Therefore, by the definition of the members of $\pi(M,B,Y)$, the result follows.
\end{proof}
In \cite{Tutte:1965}, Tutte proved that if  $M$ is binary the members  of $\pi(M,B,Y)$ are disjoint and their union is $Y$. 
We usually refer to $\pi(M,B,Y)$ as the partition of $Y$ determined by $B$. By this result, Theorem~\ref{th_ila} has the 
following useful Corollary~\ref{cor_corl1}.
\begin{corollary} \label{cor_corl1}
Let $Y$ be a cocircuit of a matroid $M$. If $M$ is binary then $\pi(M,B,Y)=\mathcal{C}^{*}(M\cto(B\cup{Y})\dto{Y})$.
\end{corollary}
Let $B_1$ and $B_2$ be two bridges of $Y$ in $M$. The bridges $B_1$ and $B_2$ are said to \emph{avoid} each other if there 
exists $S\in{\pi(M,B_1,Y)}$ and $T\in{\pi(M,B_2,Y)}$ such that $S\cup{T}=Y$; otherwise we say that $B_1$ and $B_2$ \emph{overlap} 
one another. A cocircuit $Y$ is called \emph{bridge-separable} if its bridges can be classified into two classes $U$ and $V$ such 
that no two members of the same class overlap.   Tutte in~\cite{Tutte:1965} has shown that all
cocircuits of graphic matroids are bridge-separable while if a matroid has a cocircuit which is not bridge-separable, then it 
will contain a minor isomorphic to $M^{*}(K_5)$, $M^{*}(K_{3,3})$ or $F_7^{*}$. 
Recall that by definition there is one-to-one correspondence between the family of 
edge-sets of the separates of $\Sigma\dof{Y}$ and the  family of bridges of $Y$ in $M(\Sigma)$. Suppose now that $B$ 
is a bridge of $Y$ in $M(\Sigma)$  and let $\Sigma_i$ be the component of $\Sigma\dof{Y}$ such that 
$\Sigma\dto{B} \subseteq{\Sigma_i}$. Then, if $v$ is a vertex of $V(\Sigma\dto{B})$, we denote by $C(B,v)$ the 
component of $\Sigma_i\dof{B}$ having $v$ as a vertex. Moreover, we denote by $Y(B,v)$ the set of all 
$y\in{Y}$ such that one end of $y$ in $\Sigma$ is a vertex of $C(B,v)$. 
Two well known results which are a 
consequence of the theory of bridges, is Tutte's recognition algorithm for graphic matroids in~\cite{Tutte:1960} and 
Bixby and Cunningham's efficient algorithm for testing whether a matroid is 3-connected or not in~\cite{BixCunn:79}.

\subsection{Cocircuits and Bonds} \label{subsec_corc_bonds}
Let $Y$ be a cocircuit of a connected binary signed-graphic matroid $M(\Sigma)$. 
Clearly, $Y$ is a bond of $\Sigma$ and by Theorems~\ref{th_nero1} and~\ref{th_sliii} we have that $\Sigma$ is 
2-connected and tangled. By the classification of bonds based on the nature of $\Sigma\dof Y$  presented in section~\ref{subsec_tangled_graphs},
we know that $Y$ can be one of the following types of bonds in $\Sigma$ (see Figure~\ref{fig_bonds_tangled_signed_graphs}): 
\begin{itemize}
\item[(a)] balancing bond
\item[(b)] star or unbalancing bond such that the core of $\Sigma\dof Y$ is a B-necklace
\item[(c)] star bond such that the core of $\Sigma\dof Y$ is not a B-necklace
\item[(d)] unbalancing bond such that the core of $\Sigma\dof Y$ is not a B-necklace
\end{itemize}
\begin{figure}[hbtp]
\begin{center}
\mbox{
\subfigure[balancing bond]
{
\includegraphics*[scale=0.29]{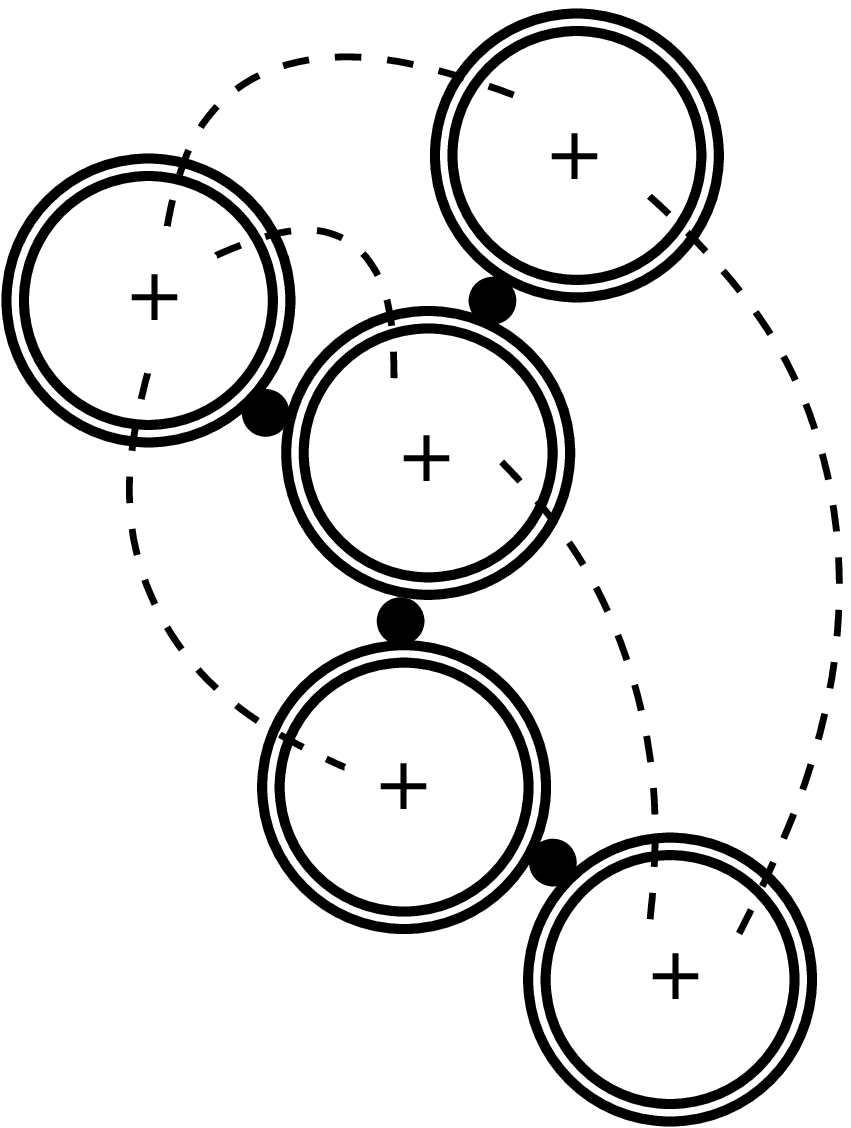}
}\quad
\subfigure[bond with a B-necklace]
{
\includegraphics*[scale=0.29]{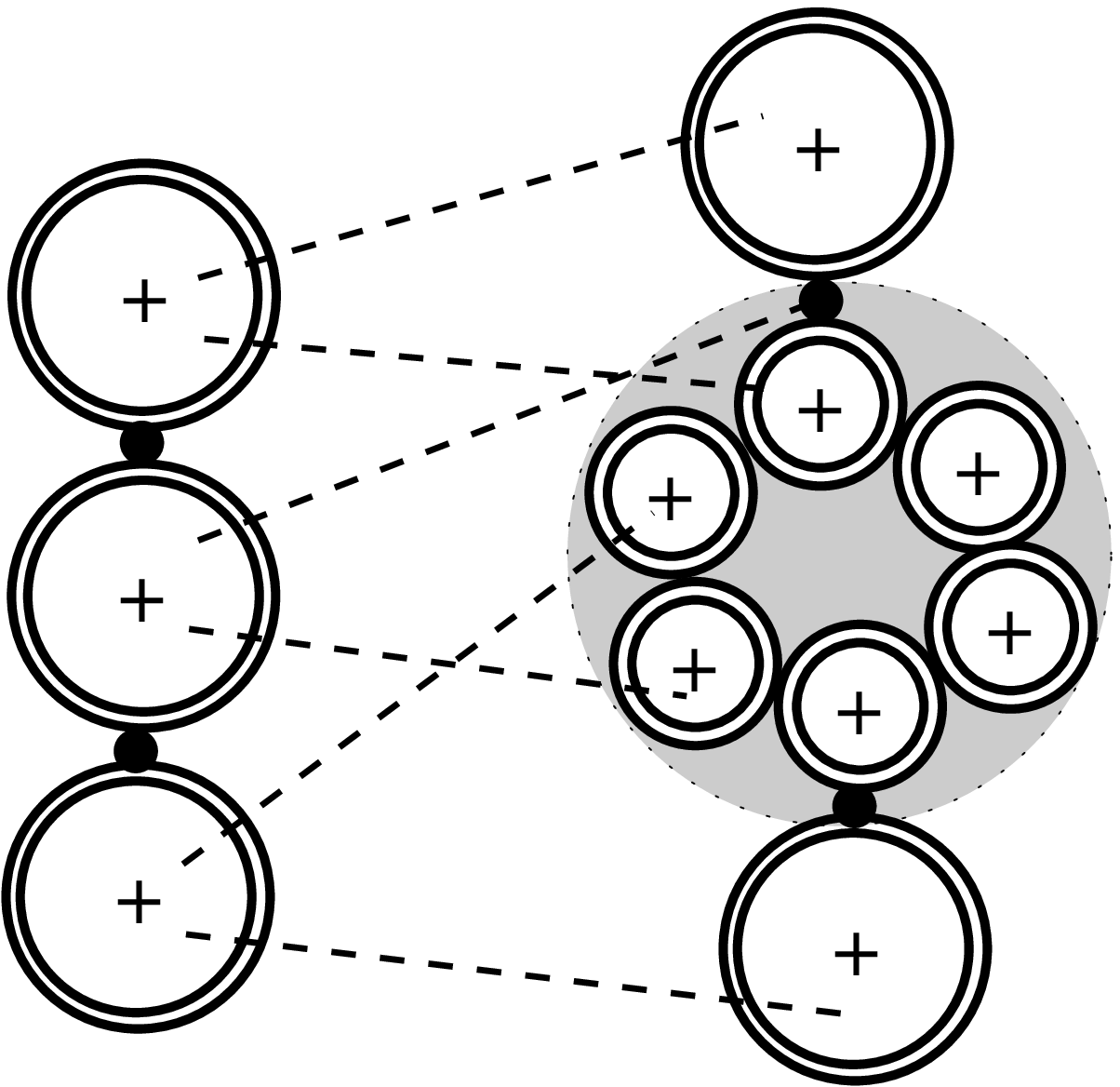}
}\quad
\subfigure[star bond]
{
\includegraphics*[scale=0.29]{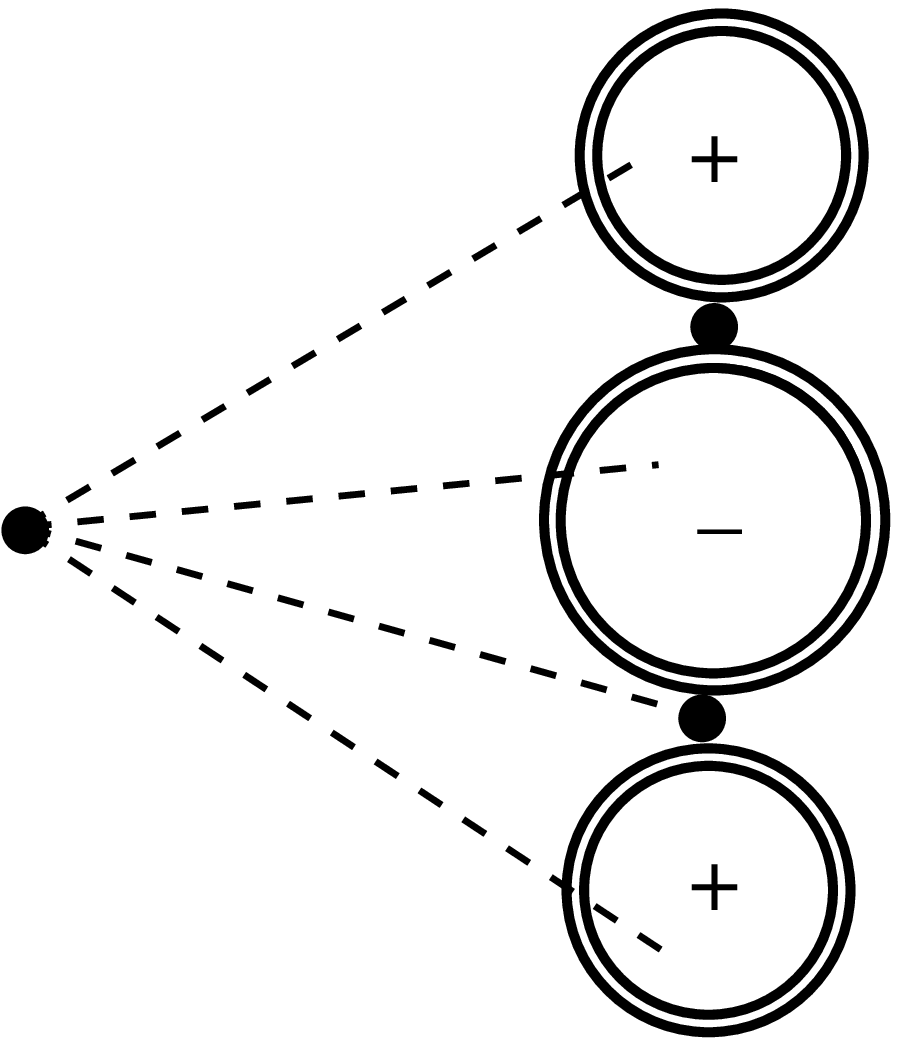}
}\quad
\subfigure[unbalancing bond]
{
\includegraphics*[scale=0.29]{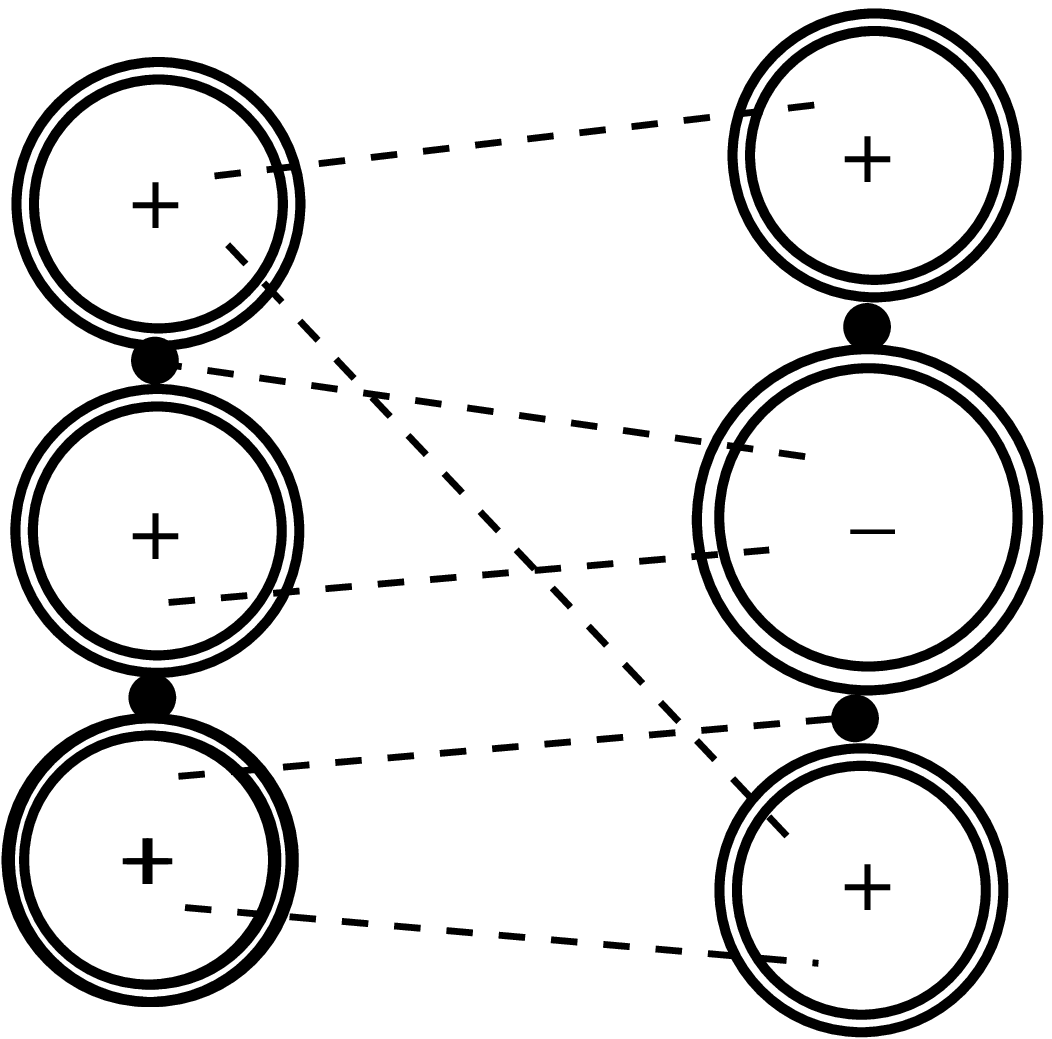}
}} \caption{Bonds in Tangled Signed Graphs}
\label{fig_bonds_tangled_signed_graphs}
\end{center}
\end{figure}
Any bridge $B$ of $Y$ in $M(\Sigma)$ will correspond to a 2-connected subgraph $\Sigma \dto B$ in $\Sigma$, 
which will be a block of $\Sigma \dof Y$. Moreover $\Sigma \dof Y$ will contain at most one unbalanced block.
Note that the only case in which a block of $\Sigma\dof Y$ does not correspond to a  separator of $M(\Sigma) \dof Y$, 
is when the block is unbalanced and a B-necklace (i.e. see (b) in Figure~\ref{fig_bonds_tangled_signed_graphs}). 
In this case the blocks within the B-necklace in the signed graph are the separators in the matroid.

We observe that if $Y$ is either of type (a) or (b), then $M(\Sigma) \dof Y$ is graphic since all of its separators 
have a balanced signed-graphic representation. Let us call {\em graphic} any cocircuit $Y$ of a binary matroid $M$ such that
$M\dof Y$ is a graphic matroid. Therefore if  $M$ is signed-graphic and $Y$ is a non-graphic cocircuit, we know that $Y$ 
will be a bond of type (c) or (d) only, in any signed graph $\Sigma$ such that $M=M(\Sigma)$.
As it turns out non-graphic cocircuits have similar structural characteristics to cocircuits of graphic matroids, and as 
it will be demonstrated in  section~\ref{subsubsec_nongraphic} they provide a means of decomposing binary signed-graphic matroids.

\subsubsection{Graphic Cocircuits}\label{subsubsec_graphic}
We know that all graphic matroids are signed-graphic. Two important theorems which associate signed-graphic matroids 
with cographic matroids and regular matroids in terms of excluded minors have been shown by Slilaty in~\cite{Slilaty:2005b}. 
Specifically, of the 35 forbidden minors for projective planar graphs 29 are non-separable. These 29 graphs, which we call 
$G_1,G_2,\ldots,G_{29}$, can be found in~\cite{Arch:81, MoTh:01}. The family of the cographic matroids of these 29 
non-separable graphs $\mathcal{M}=\{ M(G_1),M(G_2),\ldots,M(G_{29}) \}$ forms the complete list of the cographic 
excluded minors for signed-graphic matroids. 

\begin{theorem}[Slilaty~\cite{Slilaty:2005b}] \label{th_ree}
A cographic matroid $M$ is signed-graphic if and only if $M$ has no minor isomorphic to $M^{*}(G_1),\ldots,M^{*}(G_{29})$. 
\end{theorem}

Clearly, since cographic matroids are regular matroids we 
expect the list of regular excluded minors for signed-graphic matroids to contain the matroids in $\mathcal{M}$ 
and some other matroids. It is shown in~\cite{SliQinZh:09} that those other matroids are the $R_{15}$ and $R_{16}$ matroids whose binary compact representation matrices 
are the following
\begin{equation*}
A_{R_{15}}=
\left[ \begin{array}{rrrrrrrr}
1  & 0  & 1  &  0 & 0  & 0  & 0  & 1  \\
0  & 0  & 0  &  1 & 1  & 0  & 1  & 0  \\
1  & 1  & 0  &  0 & 1  & 1  & 0  & 0  \\ 
1  & 1  & 0  &  0 & 0  & 1  & 1  & 0  \\
0  & 1  & 1  &  1 & 1  & 0  & 0  & 0  \\ 
0  & 1  & 1  &  1 & 1  & 1  & 0  & 0  \\
0  & 1  & 1  &  1 & 1  & 1  & 0  & 1  
\end{array} \right], 
A_{R_{16}}=
\left[ \begin{array}{rrrrrrrr}
0  & 1  & 1  &  0 & 1  & 0  & 0  & 0  \\
0  & 0  & 0  &  0 & 1  & 1  & 1  & 0  \\
0  & 1  & 1  &  0 & 1  & 1  & 1  & 0  \\ 
0  & 0  & 0  &  1 & 1  & 1  & 0  & 0  \\
1  & 1  & 0  &  1 & 1  & 1  & 0  & 0  \\
1  & 0  & 1  &  1 & 0  & 0  & 0  & 0  \\
1  & 1  & 0  &  0 & 0  & 0  & 0  & 1  \\
1  & 0  & 1  &  0 & 0  & 0  & 0  & 1  
\end{array} \right].
\end{equation*}
Moreover, the binary excluded minors for signed-graphic matroids can be easily obtained by adding to the list of the 31 regular excluded minors of signed-graphic matroids the binary excluded minors for regular matroids (i.e. $F_7$ and $F_7^{*}$), since any binary signed-graphic matroid is also regular.
\begin{theorem}
A binary matroid $M$ is signed-graphic if and only if $M$ has no minor isomorphic to $M^{*}(G_1),\ldots,M^{*}(G_{29}), R_{15}$, $R_{16}$, $F_7$ or $F_7^{*}$. 
\end{theorem}

The following two lemmas are essential for the proof of the main result of this section which characterizes the binary 
matroids with graphic cocircuits. 

\begin{lemma} \label{lem_ll1}
If a matroid  $M$ is isomorphic to $M^{*}(G_{17})$ or  $M^{*}(G_{19})$ then for any cocircuit 
$Y\in{\mathcal{C}(M^{*})}$, the matroid $M\dof{Y}$ is graphic. 
\end{lemma}
\begin{proof}
Since $M\dof{Y}=(M^{*}/Y)^{*}$ we can equivalently show that for any circuit $Y$ of $M^{*}$, the matroid $M^{*}/Y$ is graphic. 
The matroid $M^{*}$ is cographic and thus regular. By a result of Tutte in~\cite{Tutte:1959}, a regular matroid is cographic if 
and only if it has no minor isomorphic to $M(K_5)$ or $M(K_{3,3})$. Therefore, it is enough to show that for any circuit $Y$ 
of $M(G_{17})$  ($M(G_{19})$) the matroid $M(G_{17})/Y$  ($M(G_{19})/Y$) has no minor isomorphic to $M(K_5)$ or $M(K_{3,3})$. 
Observe that $G_{17}$ is isomorphic to the graph $K_{3,5}$ and $G_{19}$ is isomorphic to $K_{4,4}\dof{e}$ where $e$ is any 
element of $K_{4,4}$. Since $M(G_{19})$ is a graphic matroid  we have that $M(G_{19})\cong{M(K_{4,4}\dof{e})}=M(K_{4,4})\dof{e}$; 
this implies that  $M(K_{4,4})$ has a minor isomorphic to $M(G_{19})$. Thus, it suffices to prove that for any circuit 
$Y_1$ ($Y_2$) of $M(K_{3,5})$  ($M(K_{4,4})$) the matroid  $M(K_{3,5})/Y_1=M(K_{3,5}/Y_1)$  ($M(K_{4,4})/Y_2=M(K_{4,4}/Y_2)$) has no 
minor isomorphic to $M(K_5)$ or $M(K_{3,3})$.

We  turn our attention to the graphs $K_{3,5}$ and $K_{4,4}$. We have that $Y_1 (Y_2)$ is a cycle of  $G_{17} (G_{19})$. By the
graphical representation of $K_{3,5}$ and  $K_{4,4}$ we can easily observe that $K_{3,5}$ and  $K_{4,4}$  have no cycle of 
cardinality less than four, since they are both 3-connected bipartite graphs. This means that $K_{3,5}/Y_{1}$ and $K_{4,4}/Y_{2}$ 
have at most five vertices each, which implies that  $M(K_{3,5}/{Y_1})$ and  $M(K_{4,4}/Y_2)$ have rank at most $4$ which is 
less than the rank of $M(K_{3,3})$. Therefore, $M(K_{3,5}/{Y_1})$ and  $M(K_{4,4}/Y_2)$ can not have a minor isomorphic to $M(K_{3,3})$. 

We now  show  that $M(K_{3,5}/{Y_1})$ and  $M(K_{4,4}/Y_2)$ can not have a minor isomorphic to $M(K_5)$. Let us suppose that 
$Y_1$ and $Y_2$ are circuits of cardinality four; thus $Y_1$ and $Y_2$ are cycles of cardinality four in $K_{3,5}$ and 
$K_{4,4}$ respectively. Observe now that  $K_{3,5}/Y_1$  and $K_{4,4}/Y_2$ are isomorphic to the 
graphs $\bar{G}$ and $\hat{G}$ respectively (see Figure~\ref{fig:g2}),  since $K_{3,5}$ and $K_{4,4}$ are complete bipartite graphs. 
Furthermore, parallel edges of a graph correspond to parallel elements in the associated graphic matroid. Therefore, any 
simple minor of $M(\bar{G})$ or $M(\hat{G})$ has at most seven or eight elements respectively. The matroid $M(K_5)$ is 
simple and has ten elements. Therefore  $M(K_5)$ can not be a minor of $M(\bar{G})=M(K_{3,5}/{Y_1})$ or $M(\hat{G})=M(K_{4,4}/Y_2)$. 
Assume now that $Y_1$ and $Y_2$ have more than four elements. Then as it was proved in the previous 
paragraph that  $M(K_{3,5}/{Y_1})$ and  $M(K_{4,4}/Y_2)$ can not have a minor isomorphic to $M(K_{3,3})$, 
we can easily show that these matroids have no minor isomorphic to $M(K_{5})$.
\begin{figure}[h] 
\begin{center}
\centering
\psfrag{G1}{\footnotesize $\bar{G}$}
\psfrag{G2}{\footnotesize $\hat{G}$}
\includegraphics*[scale=0.3]{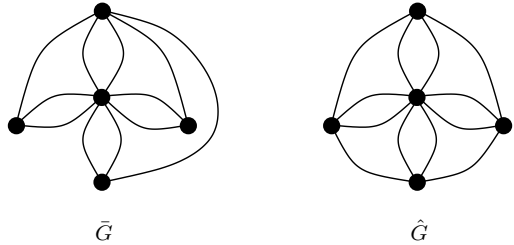}
\end{center}
\caption{The graphs $\bar{G}$ and  $\hat{G}$.}
\label{fig:g2}
\end{figure} 
\end{proof}

\begin{lemma}\label{lem_1}
If $N$ is a minor of the matroid $M$ then for any cocircuit $C_{N}\in\mathcal{C}(N^{*})$ there
exists a cocircuit $C_{M}\in\mathcal{C}(M^{*})$ such that $N\dof C_{N}$ is a minor of 
$M\dof C_{M}$. 
\end{lemma}
\begin{proof}
If $N=M\dof X / Y$ then by duality $N^{*}=M / X \dof Y$. Therefore by the definitions
of contraction and deletion of a set, we have
that for any cocircuit $C_{N}\in\mathcal{C}(N^{*})$ there exists a cocircuit $C_{M}\in\mathcal{C}(M^{*})$ such
that 
\begin{itemize}
\item[(i)] $C_{N} \subseteq C_{M}$,
\item[(ii)] $E(N) \cap C_{M} = C_{N}$,
\end{itemize}
which in turn imply that  $C_{M} - C_{N} \subseteq X$. So we have 
\[
M\dof C_{M}=  M \dof \{C_{M} - C_{N}\} \dof C_{N} \succeq N \dof C_{N}
\]
\end{proof}

We are now ready to prove the main result of this section.

\begin{theorem}\label{th_ll3}
Let $M$ be a binary matroid such that all its cocircuits are graphic. Then, $M$ is signed-graphic if and only if $M$ has no minor 
isomorphic to $M^{*}(G_{17})$, $M^{*}(G_{19})$, $F_7$ or $F_7^{*}$.
\end{theorem}
\begin{proof}
The ``only if'' part is clear because of Theorem~\ref{th_ree}.
For the ``if'' part, by way of contradiction, assume that $M$ is not signed-graphic. By Theorem~\ref{th_ree} and the straightforward fact that all cocircuits of $F_7$ and $F_7^{*}$ are graphic, $M$ must contain
a minor $N$ which is isomorphic to some matroid in the set
\[
\mathcal{M}=\{M^{*}(G_1),\ldots,M^{*}(G_{16}),M^{*}(G_{18}),
M^{*}(G_{20}),\ldots,M^{*}(G_{29}), R_{15}^{*}, R_{16}^{*}\}.
\]
By case analysis, verified also by the MACEK software~\cite{Hlileny:07},
it can be shown that for each  matroid  $M'\in \mathcal{M}$ there exists
a cocircuit $Y'\in \mathcal{C}(M'^{*})$ such that the matroid $M'\backslash{Y'}$ 
does contain an $M^{*}(K_{3,3})$ or an $M^{*}(K_{5})$ as a minor. Therefore there exists 
a cocircuit $Y_{N}\in\mathcal{C}(N^{*})$ such  that $N\backslash Y_{N}$ is not graphic. 
Therefore, by Lemma~\ref{lem_1}, there is a cocircuit $Y_{M}\in\mathcal{C}(M^{*})$ such that
$N\backslash Y_{N}$ is a minor of $M\backslash Y_{M}$.  Thus, $M\backslash Y_{M}$
is not graphic which is in contradiction with our assumption that $M$ has graphic cocircuits.
\end{proof}

\subsubsection{Non-Graphic Cocircuits}\label{subsubsec_nongraphic}
The following technical lemma is necessary for the proof of Theorem~\ref{thrm_sep81}.
\begin{lemma} \label{lem_un1}
Let $Y$ be an unbalancing bond or a star bond of a tangled signed graph $\Sigma$ such that the core of $\Sigma\backslash{Y}$ is 
not a B-necklace and let $\Sigma\dto{B}$ be an unbalanced separate of $\Sigma\backslash{Y}$. Then:
\begin{itemize}
\item [(i)] $M(\Sigma)$ is graphic, or
\item [(ii)] there exists a series of switchings on the vertices of $\Sigma$ such that all the edges of the separates 
other than $\Sigma\dto{B}$ become positive and for any $v_i\in{V(\Sigma\dto{B})}$ such that $Y(B,v_i)\neq{\emptyset}$, 
the edges of $Y(B,v_i)$ have the same sign.
\end{itemize}
\end{lemma}
\begin{proof}
Let $\bar{V}_B :=\{v_i\in{V(\Sigma\backslash.{B}})\; : \; Y(B,v_i)\neq{\emptyset}\}$. By Theorem~\ref{thrm_tangled2}, 
we have that $\Sigma\backslash{Y}$ consists of two components $\Sigma_1$ and $\Sigma_2$ and that 
there is exactly one unbalanced block in $\Sigma\backslash{Y}$, which without loss of generality we assume to be contained 
in $\Sigma_1$. Since this unbalanced block is not a B-necklace, its edge-set is a bridge $B$ of $Y$ in $M(\Sigma)$ and
therefore, $\Sigma\dto{B}$ is a separate of $\Sigma\backslash{Y}$. By Proposition~\ref{prop_sgn_graphic}, there exists a series of 
switchings on the vertices of the balanced subgraphs $C(B,v_i)$, for all $v_i\in{V(\Sigma\dto{B})}$, and at the vertices of $\Sigma_2$ such 
that all the edges in $\Sigma_1\backslash{B}$ and $\Sigma_2$ become positive. We call $\Sigma'$, $\Sigma_1'$ and $\Sigma_2'$ 
the signed graphs so-obtained from $\Sigma$, $\Sigma_1$ and $\Sigma_2$, respectively, by applying these switchings. 

For each $v_i\in{V(\Sigma'\dto{B})}$ let $Y^{+}(B,v_i)$ and  $Y^{-}(B,v_i)$ be the positive and negative edges of  
$Y(B,v_i)$, respectively, and 
\[
V_B:=\{v_i\in{\bar{V}_B}\; : \; Y^{+}(B,v_i), Y^{-}(B,v_i)\neq{\emptyset}\}. 
\]

Suppose that $|V_B|\geq{2}$, and let $V_B=\{v_1,\ldots, v_k\}$ for positive integer $k\geq 2$. In $\Sigma'\cto (B\cup{Y})$, 
since $\Sigma_1'\backslash{B}$ and $\Sigma_2'$ consists of only positive edges, every component $C(B,v_i)$ of 
$\Sigma_1'\backslash{B}$ will contract to $v_i$ and $\Sigma_2'$ will contract to a single vertex $u$. For example, 
Figure~\ref{fig_u11} depicts the graph $\Sigma'\cto ({B\cup{Y}})$ obtained from $\Sigma'$  where the dashed lines indicate
the edges of $Y$. Therefore, in $(\Sigma'\cto({B\cup{Y}}))\dto{Y}$, each $Y(B,v_i)$ with $v_i\in{V_B}$ will become a class of 
parallel edges and all the edges in $Y$ will have $u$ as an end-vertex (see Figure~\ref{fig_u11}). Thus, the following 
set $\mathcal{L}$ is a set of bonds of $(\Sigma'\cto (B\cup{Y}))\dto{Y}$:
\[
\mathcal{L}=\{\{Y^{-}(B,v_1),\ldots,Y^{-}(B,v_k)\},\{Y^{+}(B,v_1),\ldots,Y^{+}(B,v_k)\},\{Y(B,v_1)\},\ldots,\{Y(B,v_k)\}\}.
\]
\begin{figure}[h] 
\begin{center}
\centering
\psfrag{v1}{\tiny $v_1$}
\psfrag{v2}{\tiny $v_2$}
\psfrag{v3}{\tiny $v_3$}
\psfrag{-}{\footnotesize $-$}
\psfrag{+}{\footnotesize $+$}
\psfrag{CB1}{\tiny $C(B,v_1)$}
\psfrag{CB2}{\tiny $C(B,v_2)$}
\psfrag{CB3}{\tiny $C(B,v_3)$}
\psfrag{B}{\footnotesize $B$}
\psfrag{S1}{\footnotesize $\Sigma'$}
\psfrag{S2}{\footnotesize $\Sigma' \cto (B\cup Y)$}
\psfrag{S3}{\footnotesize $\Sigma' \cto (B\cup Y)\dto Y$}
\psfrag{u}{\tiny $u$}
\includegraphics*[scale=0.30]{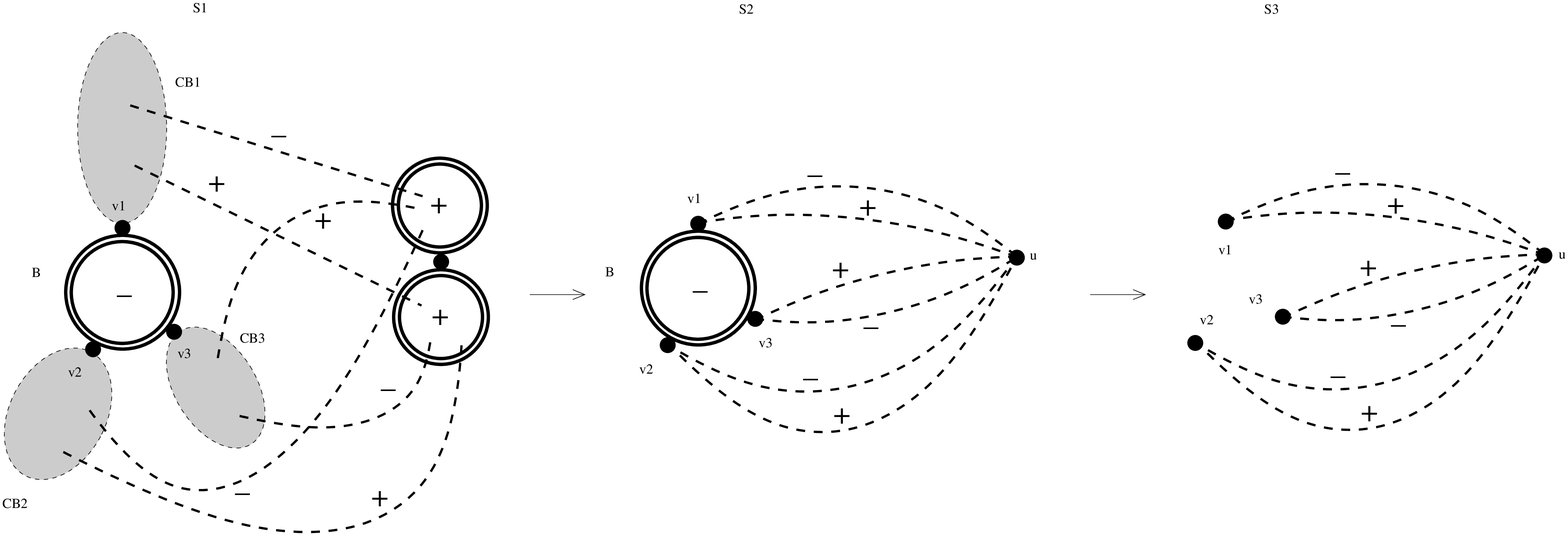}
\end{center}
\caption{}
\label{fig_u11}
\end{figure}
Therefore, by Theorems~\ref{thrm_bonds} and~\ref{th_minsig} and Proposition~\ref{prop_samematroid}, $\mathcal{L}$ is a set of cocircuits of 
$M((\Sigma'\cto(B\cup{Y}))\dto{Y})=(M(\Sigma')\cto(B\cup{Y}))\dto{Y}=(M(\Sigma)\cto(B\cup{Y}))\dto{Y}$. Since $M(\Sigma)$ is binary, 
then  by Corollary~\ref{cor_corl1}:
\[
\mathcal{C}^{*}(M(\Sigma)\cto(B\cup{Y})\dto{Y})=\pi(M(\Sigma),B,Y).
\]
But we know by (7.3) in~\cite{Tutte:1959}, that the members of $\pi(M(\Sigma),B,Y)$ should be disjoint which is a contradiction. 
Therefore, $|V_B| < {2}$.

Suppose now that $V_B=\{v\}$. If we assume that there exist $y_1\in{Y^{-}(B,v)}$ and $y_2\in{Y^{+}(B,v)}$ such that both 
$y_1$ and $y_2$ do not have $v$ as an end-vertex then the negative cycle $N$ formed by these two edges and the positive 
paths between their end-vertices in $C(B,v)$ and $\Sigma_2'$, respectively, is vertex disjoint with some negative cycle 
in the unbalanced separate $\Sigma'\dto{B}$. Since $\Sigma'$ is tangled, this can not happen and, therefore, all edges 
of $Y^{-}(B,v)$ or all edges of $Y^{+}(B,v)$ have $v$ as an end-vertex in $\Sigma'$. Furthermore, since $N$ has no vertex 
in $\Sigma\dto{B}$ other than $v$ we have that any negative cycle in $\Sigma\dto{B}$ must be incident to $v$; otherwise 
two vertex disjoint negative cycles are contained in $\Sigma'$.

Assume that $C$ is a negative cycle in $\Sigma'$ not adjacent to $v$.  Since $C$ is not contained in $\Sigma\dto{B}$ 
it contains edges from $Y$ and, furthermore, since $Y$ is an unbalancing bond the number of these edges has to be an even 
number. Say that these edges are those contained in $Y_{C}=\{y_1,y_2,\ldots,y_{2n}\}$. Arrange the edges in $Y_C$, each of 
which has one end-vertex in $\Sigma_1'$ and one in $\Sigma_2'$, as shown in Figure~\ref{fig_u14}, where the dashed lines 
are the $n$ paths between the end-vertices of the edges of $Y_C$ in $\Sigma_1'$ and $\Sigma_2'$. Since $C$ is negative 
there exist vertices $w_i$ and $w_{i+1}$ of $C$ in $\Sigma_2'$ such that the path $(w_i,y_i,u_i,\ldots,u_{i+1},y_{i+1},w_{i+1})$ 
is negative.
\begin{figure}[h] 
\begin{center}
\centering
\psfrag{u1}{\tiny $u_1$}
\psfrag{u2}{\tiny $u_2$}
\psfrag{ui}{\tiny $u_i$}
\psfrag{ui1}{\tiny $u_{i+1}$}
\psfrag{u2n1}{\tiny $u_{2n-1}$}
\psfrag{u2n}{\tiny $u_{2n}$}
\psfrag{y1}{\footnotesize $y_1$}
\psfrag{y2}{\footnotesize $y_2$}
\psfrag{yi}{\footnotesize $y_i$}
\psfrag{yi1}{\footnotesize $y_{i+1}$}
\psfrag{y2n1}{\footnotesize $y_{2n-1}$}
\psfrag{y2n}{\footnotesize $y_{2n}$}
\psfrag{w1}{\tiny $w_1$}
\psfrag{w2}{\tiny $w_2$}
\psfrag{wi}{\tiny $w_i$}
\psfrag{wi1}{\tiny $w_{i+1}$}
\psfrag{w2n1}{\tiny $w_{2n-1}$}
\psfrag{w2n}{\tiny $w_{2n}$}
\psfrag{S1}{\footnotesize $\Sigma_1$}
\psfrag{S2}{\footnotesize $\Sigma_2$}
\includegraphics*[scale=0.45]{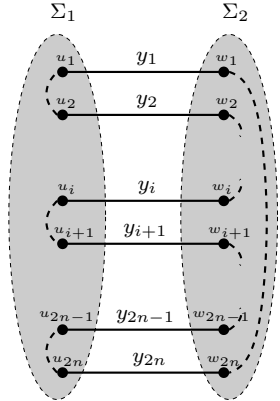}
\end{center}
\caption{An arrangement of the edges of $Y_{C}$.}
\label{fig_u14}
\end{figure}
Assume that $y_i,y_{i+1} \in {Y(B,v_l)}$ for some $v_l\in{\bar{V}_B}$. Then $v_l\neq{v}$, since $C$ is not incident to $v$. 
If $v_l\in{\bar{V}_B}-\{v\}$, then the path between $u_i$ and $u_{i+1}$ in $C(B,v_l)$ is positive or empty, which implies 
that $y_i$ and $y_{i+1}$ are of different sign. This in turn implies that $|V_B|>1$. Therefore, there exist vertices $v_1$ 
and $v_2$ in $\bar{V}_B$   such that $v_1\neq v_2 \neq v$ and $y_i\in Y(B,v_1)$ and $y_{i+1} \in Y(B,v_2)$ 
Contract the edges of $C(B,v_1)$ and $C(B,v_2)$ and switch at $v_1$ or $v_2$ if $y_i$ and $y_{i+1}$ have different signs, 
such that the path from $v_1$ to $v_2$ of $C$ contained in $\Sigma \dto{B}$ is positive and $y_i$ and $y_{i+1}$ have 
different sign. Moreover since this path, which we shall call $v_1-v_2$ path, is positive we can perform switching at 
its vertices other than $v_1$ and $v_2$ to make all of its edges positive. Contract now the positive edges of the 
$v_1-v_2$ path into a new vertex $w$ and call $\Sigma''$ the graph so-obtained.

The effect of contracting the $v_1-v_2$ path in $\Sigma'\dto{B}$ can be seen inductively if we consider the contraction 
of a single edge $e$ from the path. Consider the graph $(\Sigma'\dto{B})/e$, where $e=(w_1,w_2)$ is some positive edge 
from the  $v_1-v_2$ path and $C'$ any negative cycle of $\Sigma\dto{B}$, which we know that it is incident to $v$. If 
$e\in{C'}$, then $E(C')-{e}$ is the edge set of a negative cycle in $(\Sigma'\dto{B})/e$, which  is 
incident to $v$. If $e\notin{C'}$, then either $C'$ is still a negative cycle in $(\Sigma'\dto{B})/e$ or $C'$ is not a 
cycle any more in the graph because it is not minimal. This is because the contraction of $e$ creates two cycles $C_1'$ 
and $C_2'$ in $(\Sigma'\dto{B})/e$ which are contained in $C'$. Of the two only one is negative, and it has to be 
adjacent to $v$, since otherwise we have a negative cycle in $\Sigma'\dto{B}$ not adjacent to $v$. Since $w_1,w_2\neq{v}$, 
the contraction of $e$ into a vertex $w$ in the $2$-connected component $\Sigma'\dto{B}$ will create a $2$-connected 
unbalancing component containing $w$ and $v$ and all the negative cycles of $(\Sigma'\dto{B})/e$, and (possibly) 
$2$-connected balanced components adjacent to $w$ (see Figure~\ref{fig_u16}). The $2$-connected unbalanced component 
in  $(\Sigma'\dto{B})/e$ is not a B-necklace since $\Sigma'\dto{B}$ is not a B-necklace, i.e. the expansion of a vertex 
in a B-necklace results in a B-necklace.
\begin{figure}[h] 
\begin{center}
\centering
\psfrag{v}{\tiny $v$}
\psfrag{w1}{\tiny $w_1$}
\psfrag{w2}{\tiny $w_2$}
\psfrag{w}{\tiny $w$}
\psfrag{-}{\footnotesize $-$}
\psfrag{+}{\footnotesize $+$}
\psfrag{SB}{\footnotesize $\Sigma'\dto{B}$}
\psfrag{SBe}{\footnotesize $\Sigma'\dto{B}/e$}
\psfrag{e}{\footnotesize $e$}
\includegraphics*[scale=0.35]{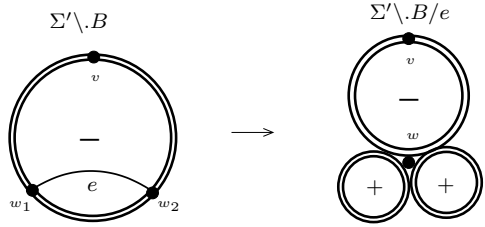}
\end{center}
\caption{Contraction of a positive edge in an unbalanced block.}
\label{fig_u16}
\end{figure}

Therefore, $Y$ is an unbalancing bond in a tangled signed graph $\Sigma''$ such that  the core of $\Sigma''\dof Y$ is not a B-necklace.
$\Sigma''\dof Y$ has an unbalanced  block $B'$ which contains vertices $v$ and $w$, where
\[
Y^{+}(B',v)\neq\emptyset \neq{Y^{-}(B',v)} \textrm{\; and \;}
\]
\[
 Y(B',w)=Y(B,v_1)\cup{Y(B,v_2)},
\]
while $y_i$ and $y_{i+1}$ are of different sign and $y_i,y_{i+1}\in{Y(B',w)}$.
But in this case $M(\Sigma'')$ is  not binary, as shown above, which contradicts the fact that $\Sigma''$ is tangled. Therefore, our 
original hypothesis that there exists negative cycle $C$ in $\Sigma'$ not adjacent to $v$ is false, which implies that $v$ is a 
balancing vertex in $\Sigma'$ and $M(\Sigma')=M(\Sigma)$ is graphic.
\end{proof}

The theorem that follows provides the graphical characterization of $\pi(M(\Sigma),B,Y)$ for a given cocircuit of a signed-graphic
matroid. 
\begin{theorem} \label{thrm_sep81}
Let $M(\Sigma)$ be a binary signed-graphic matroid and $Y$ be a star bond or an unbalancing bond of $\Sigma$ such that the core of
 $\Sigma\backslash{Y}$ is not a B-necklace. If $\Sigma\dto{B}$ is a separate of an end-graph $\Sigma_i$ of $\Sigma\backslash{Y}$ then
 $\pi(M(\Sigma),B,Y)$ is the class of all $Y(B,v)$ such that $v\in{V(\Sigma \dto{B})}$ and $Y(B,v)\neq{\emptyset}$. 
\end{theorem}
\begin{proof}

Let $\mathcal{L}=\{Y(B,v)\; : \; v\in{V(\Sigma\dto{B})} \textrm{ and } Y(B,v)\neq{\emptyset}\}$ . By Corollary~\ref{cor_corl1}, 
we know that:
\[
 \pi(M(\Sigma),B,Y) = \mathcal{C}^{*}((M(\Sigma)\cto (B\cup{Y}))\dto{Y})
\]
and thus, by Theorem~\ref{th_minsig}, we have that:
\[
 \pi(M(\Sigma),B,Y) = \mathcal{C}^{*}(M((\Sigma \cto (B\cup{Y}))\dto{Y})).
\]
Let $\mathcal{M}$ be the family of bonds of $\Sigma_b= (\Sigma \cto (B\cup{Y}))\dto{Y}$. Since there is one to one correspondence 
between the members of $\mathcal{C}^{*}(M((\Sigma \cto (B\cup{Y}))\dto{Y})$ and the bonds of $\Sigma_b$, we shall equivalently show
 that, for any bridge $B$ of $Y$ in $M(\Sigma)$, $\mathcal{L}=\mathcal{M}$. Moreover, we shall show this only for the case in which 
 $Y$ is an unbalancing bond since the proof for the case in which $Y$ is a star bond follows easily.

By Theorem~\ref{thrm_tangled2}, the signed graph $\Sigma\backslash{Y}$ will consist of two components 
$\Sigma_1$ and $\Sigma_2$ and contain exactly one unbalanced block. Without loss of generality, we assume that this unbalanced 
block is contained in $\Sigma_1$. By Proposition~\ref{prop_sgn_graphic}, since $C(B,v)$ is balanced for any $v\in{V(\Sigma \dto{B})}$ and 
$\Sigma_2$ is balanced, there exists a series of switchings on the vertices of $\Sigma_1\backslash{B_1}$  and $\Sigma_2$ such that 
all the edges in  $\Sigma_1\backslash{B_1}$  and $\Sigma_2$ become positive. We call $\Sigma'$, $\Sigma_1'$, $\Sigma_2'$ and 
$\Sigma_b'$ the graphs so-obtained from $\Sigma$, $\Sigma_1$, $\Sigma_2$ and $\Sigma_b$, respectively,  by applying these 
switchings. Figure~\ref{fig:exam_unbal} depicts an example signed graph $\Sigma'$, where the dashed edges are the edges of 
the unbalancing bond $Y$.
\begin{figure}[h] 
\begin{center}
\centering
\psfrag{S1}{\footnotesize $\Sigma'_1$}
\psfrag{S2}{\footnotesize $\Sigma'_2$}
\psfrag{S}{\footnotesize $\Sigma':$}
\psfrag{1}{\footnotesize $1$}
\psfrag{2}{\footnotesize $2$}
\psfrag{3}{\footnotesize $3$}
\psfrag{4}{\footnotesize $4$}
\psfrag{5}{\footnotesize $5$}
\psfrag{6}{\footnotesize $6$}
\psfrag{7}{\footnotesize $7$}
\psfrag{B1}{\footnotesize $B_1$}
\psfrag{B2}{\footnotesize $B_2$}
\psfrag{B3}{\footnotesize $B_3$}
\psfrag{B4}{\footnotesize $B_4$}
\psfrag{B5}{\footnotesize $B_5$}
\psfrag{B6}{\footnotesize $B_6$}
\psfrag{B0}{\footnotesize $B_0$}
\psfrag{v1}{\tiny $v_1$}
\psfrag{v2}{\tiny $v_2$}
\psfrag{v3}{\tiny $v_3$}
\psfrag{v4}{\tiny $v_4$}
\psfrag{v5}{\tiny $v_5$}
\psfrag{v6}{\tiny $v_6$}
\psfrag{v7}{\tiny $v_7$}
\psfrag{v8}{\tiny $v_8$}
\psfrag{v9}{\tiny $v_9$}
\psfrag{v10}{\tiny $v_{10}$}
\includegraphics*[scale=0.40]{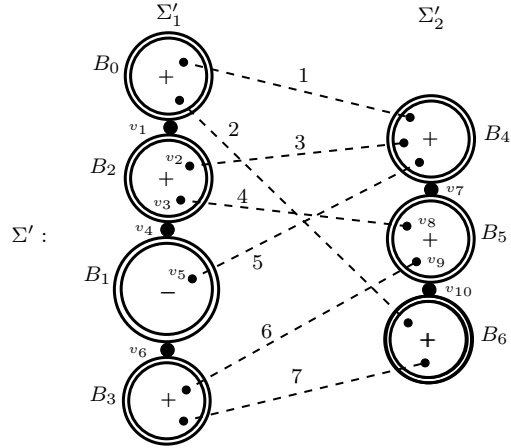}
\end{center}
\caption{An example tangled $\Sigma'$  where the dashed lines represent the edges of an unbalancing bond.}
\label{fig:exam_unbal}
\end{figure}

We classify the bridges of $Y$ in $M(\Sigma')=M(\Sigma)$ in three categories based on the form of the corresponding separates 
in $\Sigma'\backslash{Y}$. Specifically, a bridge $B$ of $Y$ in $M(\Sigma)$ falls in:
\begin{itemize}
 \item {\bf Category 1}, if the separate $\Sigma'\dto{B}$ of $\Sigma'\backslash{Y}$ is a balanced block in $\Sigma_1'$,
 \item {\bf Category 2}, if the separate $\Sigma'\dto{B}$ of $\Sigma'\backslash{Y}$ is a balanced block in $\Sigma_2'$, 
 \item {\bf Category 3}, if the separate $\Sigma'\dto{B}$ of $\Sigma'\backslash{Y}$ is the unbalanced block in $\Sigma_1'$
\end{itemize}
In what follows  we shall show that $\mathcal{L}=\mathcal{M}$ for any bridge $B$ of each category. 
We have the following three cases.\\
\noindent
{\bf Case 1:} $B$ is a bridge of $Y$ in $M(\Sigma)$ which belongs to Category 1. Initially, we shall describe the effect of the 
series of contractions and deletions in $\Sigma'$ resulting in $\Sigma_b'$. Let $X$ be the set of common vertices of $\Sigma'\dto{B}$ 
and $\Sigma_1'\backslash{B}$. Clearly, there exists an $x_j\in{X}$  such that $C(B,x_j)$ contains the unbalanced block of $\Sigma_1'$. 
The signed graph $\bar{\Sigma}'=\Sigma'\backslash{E(\Sigma_2)}$ is the graph obtained from $\Sigma'$ by contracting $\Sigma_2'$ 
into a single vertex $u$ and replacing the end-vertex of each edge $Y$ in $\Sigma_2'$ by $u$. Furthermore, the signed graph 
$\Sigma''=\bar{\Sigma}'/C(B,x_j)$, which contains $\Sigma'_b$ as a minor,  is obtained from $\bar{\Sigma}'$ by deleting $C(B,x_j)$ 
and replacing every edge in $Y(B,x_j)$ by a half-edge incident to $u$. In $\Sigma''$ we contract all $C(B,x_i)$ with $i\neq{j}$ and 
call $\Sigma_c'$ the signed graph so-obtained. Then $\Sigma_c'=\Sigma'\cto{(B\cup{Y})}$ and every edge of $Y(B,x_i)$ with $i\neq{j}$ 
has one end-vertex being $u$ and the other being $x_i$ in $\Sigma'_c$. Thus,  for any $v\in{V(\Sigma'\dto{B})-\{x_j\}}$ such that 
$Y(B,v)\neq{\emptyset}$, $Y(B,v)$ is a set of parallel edges incident to $u$ and $v$ in $\Sigma_b'=\Sigma_c'\backslash{B}$, while 
all the edges in $Y(B,x_j)$ are half-edges incident with $u$ in  $\Sigma_b'$ (see (a) in Figure~\ref{fig_exam_unbal_cases}). Furthermore, for 
any $v\in{V(\Sigma'\dto{B})-\{x_j\}}$ such that $Y(B,v)\neq{\emptyset}$, the edges of  $Y(B,v)$ must be of the same sign, since 
otherwise $\Sigma'$ would have two vertex disjoint negative cycles contradicting the fact that $\Sigma'$ is tangled. Thus, any 
$Y(B,v)\neq{\emptyset}$ is a bond of $\Sigma'_b$. This result and the fact that the signed graphs $\Sigma_b$ and $\Sigma'_b$ have equal classes of bonds imply that $\mathcal{L}$ is contained in $\mathcal{M}$. Finally, if $\Sigma_b'$ had a bond which was not equal to some $Y(B,v)$ then it would 
have two bonds with a common element and thus, by Corollary~\ref{cor_corl1}, $M(\Sigma')$ would not be regular. By Theorem~\ref{th_sliii},  this contradicts the fact that $\Sigma'$ is 
tangled and thus, $\mathcal{L}=\mathcal{M}$. \\
\noindent
{\bf Case 2:} $B$ is a bridge of $Y$ in $M(\Sigma)$ which belongs to Category 2. Since $\Sigma_2'$ consists of positive edges and 
$\Sigma_1'$ contains a negative cycle, for any $v\in{V(\Sigma'\dto{B})}$ such that $Y(B,v)\neq{\emptyset}$, $Y(B,v)$ will be a set 
of  half-edges incident with $v$ in $\Sigma'\cto({B\cup{Y}})$. Thus, the edges of each $Y(B,v)$ will form a bond of $\Sigma_b'$ 
(see (b) in Figure~\ref{fig_exam_unbal_cases}) which implies  that $\mathcal{L}$ is contained in $\mathcal{M}$. Furthermore, $\Sigma_b'$ has no other bonds, since otherwise it should have two bonds having at least one common edge. This would imply that $M(\Sigma'_b)$ would have two cocircuits which have a common element and thus, by Corollary~\ref{cor_corl1}, $M(\Sigma')$ would not be regular. By Theorem~\ref{th_sliii},  this is in contradiction with the fact that $\Sigma'$ is tangled and thus, $\mathcal{L}=\mathcal{M}$. \\
\begin{figure}[hbtp]
\begin{center}
\centering
\psfrag{1}{\tiny $1$}
\psfrag{2}{\tiny $2$}
\psfrag{3}{\tiny $3$}
\psfrag{4}{\tiny $4$}
\psfrag{5}{\tiny $5$}
\psfrag{6}{\tiny $6$}
\psfrag{7}{\tiny $7$}
\psfrag{v1}{\tiny $v_1$}
\psfrag{v2}{\tiny $v_2$}
\psfrag{v3}{\tiny $v_3$}
\psfrag{v4}{\tiny $v_4$}
\psfrag{v5}{\tiny $v_5$}
\psfrag{v6}{\tiny $v_6$}
\psfrag{v7}{\tiny $v_7$}
\psfrag{v8}{\tiny $v_8$}
\psfrag{v9}{\tiny $v_9$}
\psfrag{v10}{\tiny $v_{10}$}
\psfrag{u}{\tiny $u$}
\psfrag{S1}{\tiny $\Sigma' \cto (B_2\cup{Y})\dto{Y}$}
\psfrag{S2}{\tiny $\Sigma' \cto (B_5\cup{Y})\dto{Y}$}
\psfrag{S3}{\tiny $\Sigma' \cto (B_1\cup{Y})\dto{Y}$}
\mbox{
\subfigure[Case~1]
{
\includegraphics*[scale=0.29]{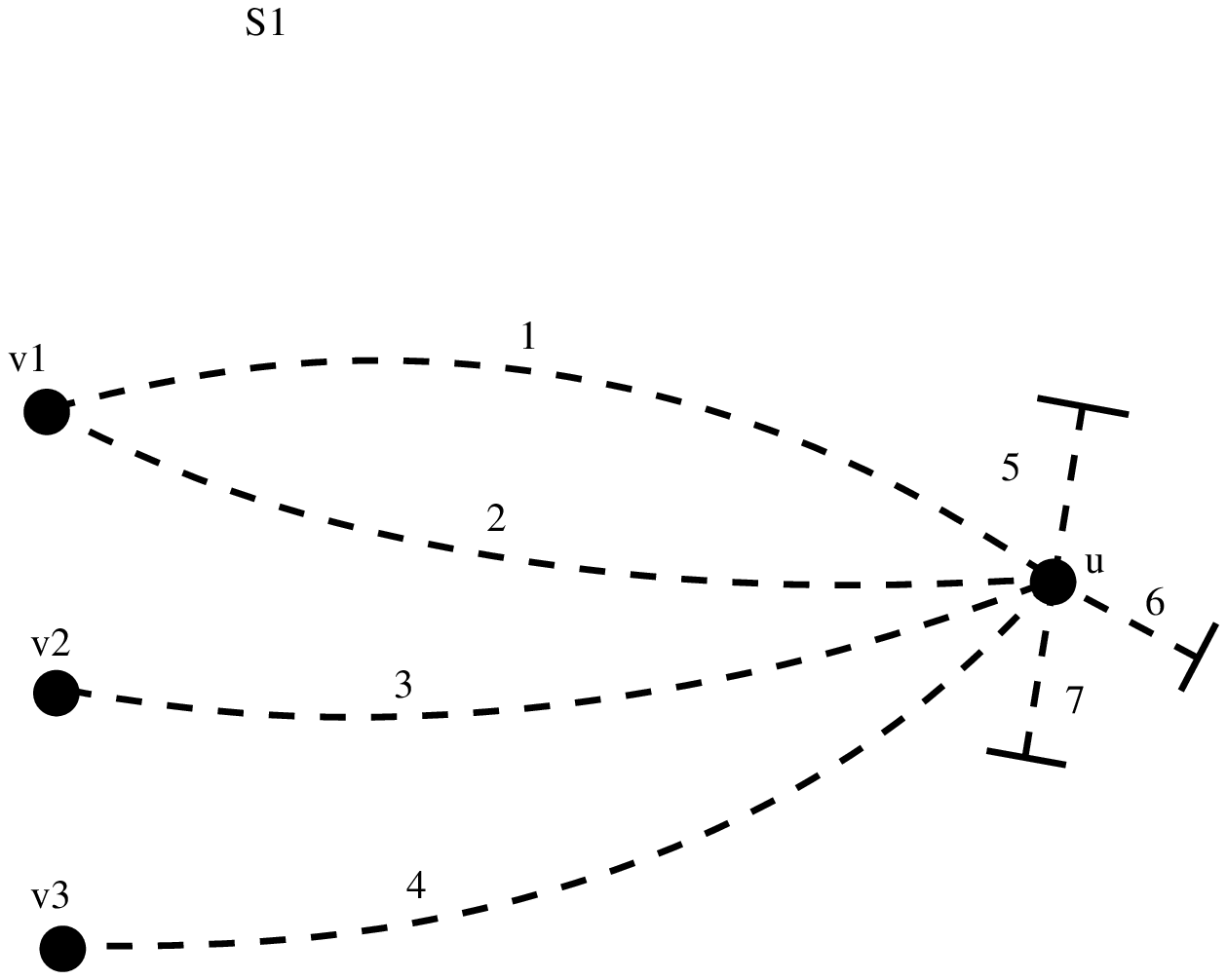}
}\quad
\subfigure[Case~2]
{
\includegraphics*[scale=0.29]{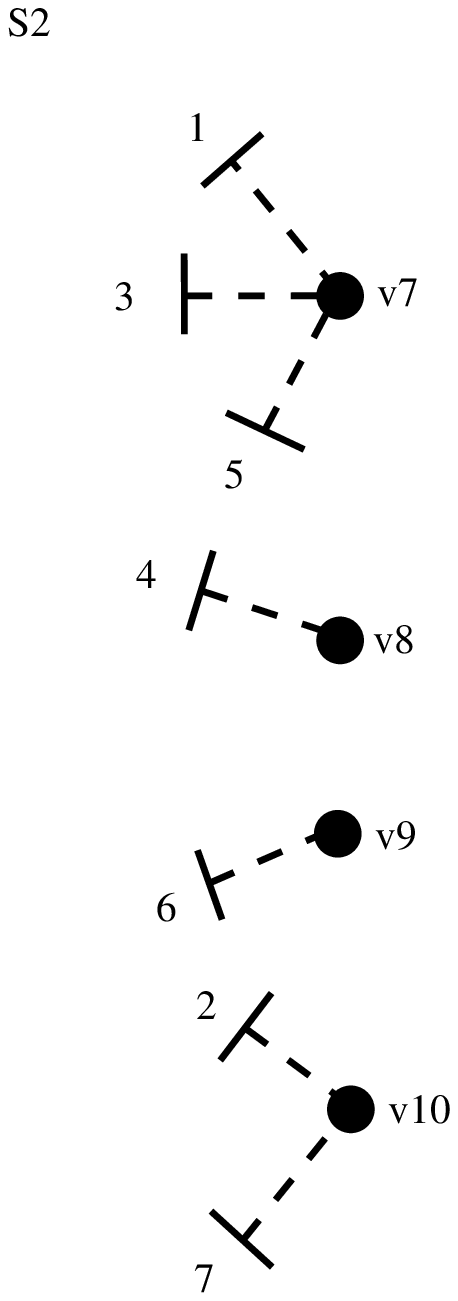}
}\quad
\subfigure[Case~3]
{
\includegraphics*[scale=0.29]{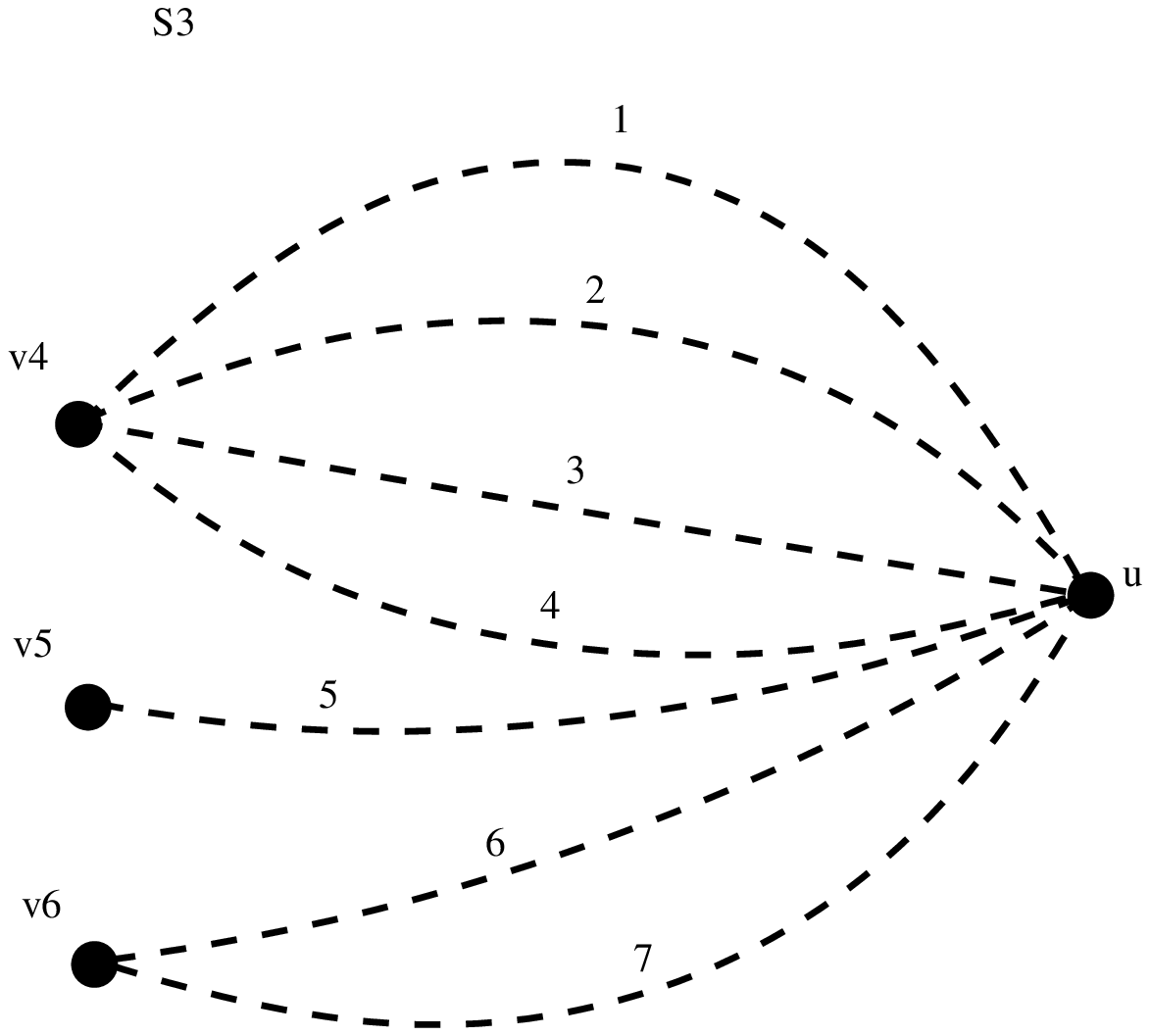}
}} \caption{Three different cases}
\label{fig_exam_unbal_cases}
\end{center}
\end{figure}
\noindent
{\bf Case 3:} $B$ is a bridge of $Y$ in $M(\Sigma)$ which belongs to Category 3. Since both $\Sigma_2'$ and $\Sigma_1'\backslash{B}$ 
consist of positive edges, the graph $\Sigma'\cto (B\cup{Y})$ is obtained from $\Sigma'$ by contracting $\Sigma_2'$ to a 
vertex $u$ and by contracting each $C(B,v)$ (where $v\in{V(\Sigma'\dto{B})}$) to $v$. Therefore, the edges of each $Y(B,v)$ become incident with 
$u$ and $v$ which implies that the edges of each $Y(B,v)$ are parallel edges in $\Sigma'_b$ (see (c) in Figure~\ref{fig_exam_unbal_cases}). 
Furthermore, by Lemma~\ref{lem_un1} and since $M(\Sigma)$ is not graphic, each $Y(B,v)$ in $\Sigma_b$ consists of edges of the 
same sign. Thus, each $Y(B,v)$ is a bond of $\Sigma'_b$ which implies that $\mathcal{L}$ is contained in $\mathcal{M}$. Finally, $\Sigma_b'$ has no other bonds, since otherwise it should  have two bonds having at least one common edge. This implies that $M(\Sigma')$ would have two cocircuits having a common element and thus, by Corollary~\ref{cor_corl1}, $M(\Sigma')$ would not be regular. By Theorem~\ref{th_sliii},  this contradicts the fact that $\Sigma'$ is tangled and thus, $\mathcal{L}=\mathcal{M}$.
\end{proof}

It turns out that star bonds or unbalancing bonds whose deletion does not result in the formation of a B-necklace, are always
bridge-separable in the corresponding signed-graphic matroid.
\begin{theorem} \label{th_louloudi}
Let $Y$ be a cocircuit of a binary signed-graphic and non graphic matroid $M(\Sigma)$. If $Y$ is a star bond 
or an unbalancing bond of $\Sigma$ such that the core of $\Sigma\dof{Y}$ is not a B-necklace then $Y$ is a bridge-separable 
cocircuit of $M(\Sigma)$.
\end{theorem}
\begin{proof}
Let $Y$ be a star bond or an unbalancing bond of $\Sigma$ such that the core of $\Sigma\dof{Y}$ is not a B-necklace. By 
Theorem~\ref{thrm_tangled2}, $\Sigma\dof{Y}$ consists of two components which we call $\Sigma_1$ and $\Sigma_2$. We arrange the 
bridges of $Y$ in $M(\Sigma)$ in two classes $T$ and $U$ such that a bridge $B_i$ is in $T (U)$ if $\Sigma\dto{B_i}$ is a 
separate of $\Sigma_1 (\Sigma_2)$. Suppose now that two bridges $B_1$ and $B_2$ of $T$ overlap. Then $\Sigma\dto{B_1}$ and 
$\Sigma\dto{B_2}$ are separates of $\Sigma_1$. Thus there exist vertices $v_1$ of $\Sigma\dto{B_1}$ and $v_2$ of 
$\Sigma\dto{B_2}$ such that $\Sigma\dto{B_2}$ is a subgraph of $C(B_1,v_1)$ and $\Sigma\dto{B_1}$ is a subgraph of $C(B_2,v_2)$. 
Furthermore, every vertex of $V(\Sigma_1)$ is a vertex of $C(B_1,v_1)$ or $C(B_2,v_2)$, therefore we have that 
$Y(B_1,v_1)\cup Y(B_2,v_2)=Y$. Thus, by Theorem~\ref{thrm_sep81}  we can find some $K\in\pi(M(\Sigma),B_1,Y)$ and 
$J\in\pi(M(\Sigma),B_2,Y)$ such that $K\cup{J}=Y$ . This is in contradiction with our assumption that $B_1$ and $B_2$ overlap 
and the result follows.
\end{proof}

As the next result demonstrates, balanced bridges of non-graphic cocircuits result in $Y$-components which are
graphic matroids.
\begin{lemma} \label{lem_balanced=graphic}
Let $M(\Sigma)$ be a binary signed-graphic matroid, $Y$ a non-graphic cocircuit and $B$ a bridge of $Y$ in $M(\Sigma)$.
If $\Sigma \dto B$ is balanced then $M(\Sigma) \cto (Y\cup B)$ is graphic.
\end{lemma}
\begin{proof}
Since $Y$ is a non-graphic cocircuit, $M(\Sigma)$ is not graphic and by  Theorem~\ref{th_sliii} $\Sigma$ is 
a tangled signed graph. Moreover, $Y$ will be either a star or unbalancing bond in $\Sigma$ such that the
core of $\Sigma \dof Y$ is not a B-necklace. It suffices to examine the case where $Y$ is an unbalancing bond.

Let $B^{+}$ be any bridge of $Y$ such that $\Sigma \dto B^{+}$ is balanced, while $B^{-}$ be the bridge corresponding to the
unique unbalanced block of $\Sigma \dof Y$.
Perform switchings in the vertices of $\Sigma$ such that all the edges in the balanced blocks of $\Sigma \dof Y$ become positive. 
If $B^{+}$ is in the balanced component of $\Sigma\dof Y$, contract any other balanced block to obtain 
$\Sigma \cto (B^{+} \cup B^{-} \cup Y )$ (see Figure~\ref{fig:balanced_graphic_1}). Contracting the edges of the 
unbalanced block $B^{-}$, where if an edge is negative switch one of its end-vertices, will result in one or more negative 
loops since this block contains negative cycles. Therefore, if we contract these negative loops according to the 
definition of contraction in signed graphs given in section~\ref{subsec_signed_graphs}, we get the signed graph
$\Sigma \cto (B^{+} \cup Y )$ where the only negative cycles are the half edges of $Y$. Therefore by 
Proposition~\ref{prop_sgn_graphic} $M(\Sigma \cto (B^{+} \cup Y ) ) = M(\Sigma) \cto (B^{+} \cup Y )$  is graphic.
\begin{figure}[h] 
\begin{center}
\centering
\psfrag{S1}{\footnotesize $\Sigma \cto (B^{+} \cup B^{-} \cup Y )$}
\psfrag{S3}{\footnotesize $\Sigma \cto (B^{+} \cup Y )$}
\psfrag{B+}{\footnotesize $B^{+}$}
\psfrag{B-}{\footnotesize $B^{-}$}
\includegraphics*[scale=0.30]{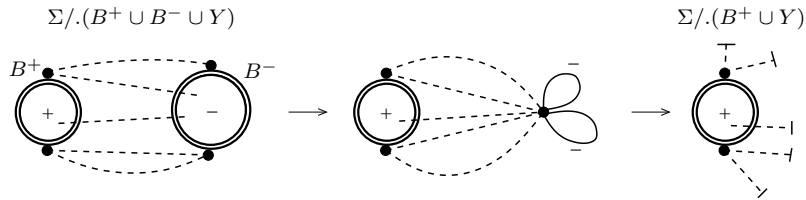}
\end{center}
\caption{$B^{+}$ in the balanced component of $\Sigma \dof Y$.}
\label{fig:balanced_graphic_1}
\end{figure}
If $B^{+}$ is in the unbalanced component of $\Sigma\dof Y$ the argument is similar. Contract again any other 
balanced block to obtain $\Sigma \cto (B^{+} \cup B^{-} \cup Y )$ (see Figure~\ref{fig:balanced_graphic_2}). 
Contraction now of the edges in the unbalanced block $B^{-}$ may result in changing the sign of the edges in 
$B^{+}$ which are adjacent to
the unique vertex of attachment $v$, while these edges will be become half edges upon deletion of the negative
loops. Therefore $\Sigma \cto (B^{+} \cup Y )$ will contain a balanced component $\bar{B}$, which is not 
necessarily 2-connected, and a number of half edges from $Y$ and $B^{+}$. If $\Sigma \cto (B^{+} \cup Y )$
contains a negative cycle $C$ other than the half edges, then $C$ would be a negative cycle in $\Sigma$ which
is disjoint from $v$, and thereby vertex disjoint with any negative cycle in $B^{-}$ implying that $\Sigma$
is not tangled. 
\begin{figure}[h] 
\begin{center}
\centering
\psfrag{S1}{\footnotesize $\Sigma \cto (B^{+} \cup B^{-} \cup Y )$}
\psfrag{S2}{\footnotesize $\Sigma \cto (B^{+} \cup Y )$}
\psfrag{B+}{\footnotesize $B^{+}$}
\psfrag{B-}{\footnotesize $B^{-}$}
\psfrag{Bbb}{\footnotesize $\bar{B}$}
\psfrag{v}{\tiny $v$}
\includegraphics*[scale=0.30]{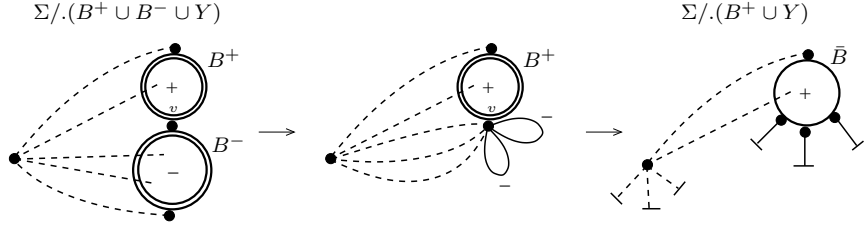}
\end{center}
\caption{$B^{+}$ in the unbalanced component of $\Sigma \dof Y$.}
\label{fig:balanced_graphic_2}
\end{figure}
\end{proof}

Theorem~\ref{th_tu84} is an extension of a result of Tutte in~\cite{Tutte:1959} (Theorem 8.4) regarding graphic matroids.  
It shows that given a signed graphic matroid $M$ and some non-graphic cocircuit $Y$ with no two overlapping bridges, there exists a 
signed graph representation where $Y$ is the star of a vertex. It is an important structural result that will be used in
the decomposition Theorem~\ref{thrm_decomposition}.
\begin{theorem} \label{th_tu84}
Let $Y$ be a non-graphic cocircuit of a connected binary signed-graphic matroid $M$ such that no two bridges of
 $Y$ in $M$ overlap. Then  there exists a $2$-connected signed graph $\Sigma$ where $Y$ is the star of a vertex 
$v\in{V(\Sigma)}$ and $M=M(\Sigma)$.
\end{theorem}
\begin{proof}
By Theorem~\ref{th_sliii} there exists a tangled 2-connected signed graph $\Sigma$ where $M=M(\Sigma)$ and $Y$ is either 
a star bond or an unbalancing bond such that the core of $\Sigma \dof Y$ is not a B-necklace. If $Y$ is a star bond there is nothing
to prove. 

Let $Y$ be an unbalancing bond of $\Sigma$, while  $\Sigma_{1}$ and $\Sigma_{2}$ the two non-empty components  of $\Sigma \dof Y$,
where one of them will contain  the unique  unbalanced block corresponding to bridge, say $B^{-}$, of $Y$ (see Theorem~\ref{thrm_tangled2}). 
Furthermore, assume that we have performed switchings such that only $Y$ and $B^{-}$ may contain edges with
negative sign. In what follows we will show that there exist disjoint 2-separations in $\Sigma$, such that by reversing on the defining vertices we can reduce the
size of one of the components of $\Sigma \dof Y$ by one separate at a time. Moreover it will be proved that the aforementioned reversings
do not alter the signed graphic matroid $M(\Sigma)$ .

Fix an arbitrary bridge $B_{0}$ of $Y$ in $M(\Sigma)$  where $\Sigma \dto B_{0}$ is a separate of $\Sigma_{2}$. For any bridge $B_{1}$ of
$Y$ such that $\Sigma \dto B_{1}$ is a separate of $\Sigma_{1}$, we know by Theorem~\ref{thrm_sep81} that there exist $v_{0} \in V(\Sigma \dto B_{0})$ and 
$v_{1} \in V(\Sigma \dto B_{1})$ such that 
\begin{equation}\label{eq_part_Y_1}
Y(B_{0}, v_{0}) \cup Y(B_{1}, v_{1}) = Y. 
\end{equation}
Choose $B_{1}, v_{1}$ and $v_{0}$ such that the number of edges of $C(B_{1},v_{1})$ is the least possible. If
$\Sigma_{i}$ the component of $\Sigma \dof Y$ such that $\Sigma \dto B \subseteq \Sigma_{i}$, define $F(B,v) := \Sigma_{i} \dof E( C(B,v) )$. 
Let $B_{1},B_{2},\ldots, B_{k}$ be the bridges of $Y$ such that $v_{1} \in V(\Sigma \dto B_{i})$ for $i=1,\ldots,k$, and consider any $B_{i}$. We know that
there exist $w \in V(\Sigma \dto B_{i})$ and $v \in V(\Sigma \dto B_{0})$ such that 
\begin{equation}\label{eq_part_Y_2}
Y(B_{i}, w) \cup Y(B_{0}, v) = Y. 
\end{equation}
We will snow that $w=v_{1}$ for all $i=1,\ldots,k$. From \eqref{eq_part_Y_1} and the fact that $\Sigma$ is 2-connected, we can deduce that there exists
at least one edge $e\in Y$ with one end-vertex in $F(B_{1},v_{1})$ and the other end-vertex in $C(B_{0},v_{0})$. Suppose now that $w \neq v_{1}$.  Then
$e \notin Y(B_{i},w)$ which implies that $v=v_{0}$ for \eqref{eq_part_Y_2} to be true. This contradicts the choice of $B_{1}, v_{1}$ and $v_{0}$ since
$E( C(B_{i},w) ) \subset E( C(B_{1},v_{1}) )$. The situation for $k=4$ is depicted in Figure~\ref{fig:2_separations}.

\begin{figure}[h] 
\begin{center}
\centering
\psfrag{S1}{\footnotesize $\Sigma_{1}$}
\psfrag{S2}{\footnotesize $\Sigma_{2}$}
\psfrag{w1}{\tiny $v_0$}
\psfrag{w}{\tiny $v$}
\psfrag{v1}{\tiny $v_1$}
\psfrag{B}{\tiny $B_0$}
\psfrag{B1}{\tiny $B_1$}
\psfrag{B2}{\tiny $B_2$}
\psfrag{B3}{\tiny $B_3$}
\psfrag{B4}{\tiny $B_4$}
\psfrag{Yw}{\tiny $Y(B_{0}, v)$}
\psfrag{FB1}{\tiny $F(B_1, v_1)$}
\psfrag{FB2}{\tiny $F(B_2, v_1)$}
\psfrag{FB3}{\tiny $F(B_3, v_1)$}
\psfrag{FB4}{\tiny $F(B_4, v_1)$}
\psfrag{CBw1}{\tiny $C(B_0, v_0)$}
\psfrag{CB1}{\tiny $C(B_1, v_1)$}
\psfrag{e}{\tiny $e$}
\includegraphics*[scale=0.40]{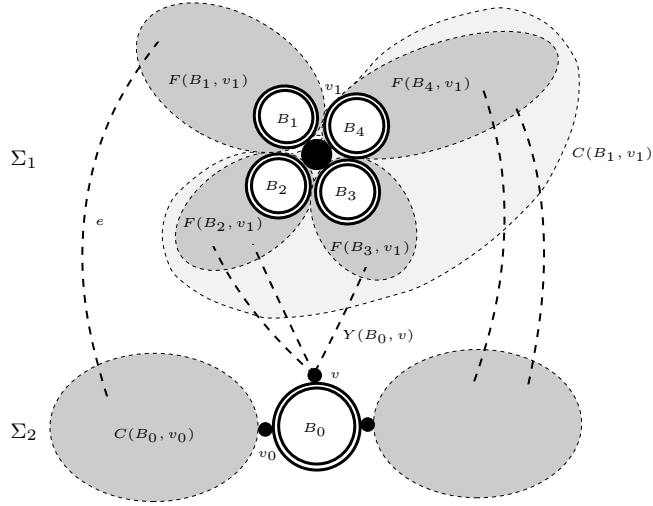}
\end{center}
\caption{Groups of bridges as 2-separations in $\Sigma$.}
\label{fig:2_separations}
\end{figure}
Partition $\{ B_{1},B_{2},\ldots, B_{k} \}$ into groups of bridges where $Y(B_{0}, v)$  from \eqref{eq_part_Y_2} is common to all bridges in a group. 
For example in Figure~\ref{fig:2_separations}, both $B_{2}$ and $B_{3}$ have $Y(B_{0}, v)$ in \eqref{eq_part_Y_2}. 
Each such
group $Q(v)$ defines a 2-separation in $\Sigma$, with defining vertices $v$ and $v_1$, and a partition $\{T, E(\Sigma) \dof T \}$ of $E(\Sigma)$ where
\begin{equation*}
T = E ( C(B_{0}, v) ) \cup Y(B_{0},v) \cup \left( \bigcup_{B_{i} \in Q(v)} E( F(B_{i}, v_{1} ) ) \right).
\end{equation*}
This is so, since for any $B_{i} \in Q(v)$ for $i=1,\ldots, k$, there are no edges of $Y$ with one end-vertex in $F(B_{i},v_{1})$ and the other in $F(B_{0},v)$ 
due to the avoidance between $B_{i}$ and $B_{0}$. Therefore by reversing about $v_{1}$ and $v$ for every group of bridges $Q(v)$, we can 
create a signed graph $\Sigma'$ with $Y$ an unbalancing bond such that one of the components of $\Sigma' \dof Y$ is $\Sigma_{2}$ with 
$B_{0}$ replaced by a vertex, while the other component is $\Sigma_{1}$ with $v_{1}$ replaced by $B_{0}$.

It remains to be shown that the reversings mentioned above do not alter $M(\Sigma)$. For every group of bridges $Q(v)$ we can identify 
three cases (see Figure~\ref{fig_2-separations}): 
(a) $C(B_{0},v)= \emptyset = \bigcap_{B_{i}\in Q(v)} C(B_{i},v_{1})$, 
(b) $C(B_{0},v)= \emptyset$ and $\bigcap_{B_{i}\in Q(v)} C(B_{i},v_{1}) \neq \emptyset$, and 
(c) $C(B_{0},v) \neq \emptyset \neq \bigcap_{B_{i}\in Q(v)} C(B_{i},v_{1})$. 
We will show that  reversing $\Sigma$ about $\{v_{1}, v\}$ produces a signed graph $\Sigma'$ such that $M(\Sigma') = M(\Sigma)$ in every case. 
\begin{figure}[hbtp]
\begin{center}
\centering
\psfrag{v1}{\tiny $v_1$}
\psfrag{v2}{\tiny $v$}
\psfrag{B1}{\tiny $B_{i}$}
\psfrag{B2}{\tiny $B_{0}$}
\psfrag{Bk}{\tiny $B_{k}$}
\psfrag{S1}{\footnotesize $\Sigma_{1}$}
\psfrag{S2}{\footnotesize $\Sigma_{2}$}
\psfrag{Y1}{\tiny $\bigcap_{B_{i}\in Q(v)} Y(B_{i},v_{1}) $}
\psfrag{Y2}{\tiny $Y(B_{0},v)$}
\psfrag{C1}{\tiny $\bigcap_{B_{i}\in Q(v)} C(B_{i}, v_{1})$}
\psfrag{C2}{\tiny $C(B_{0}, v)$}
\mbox{
\subfigure[]
{
\includegraphics*[scale=0.265]{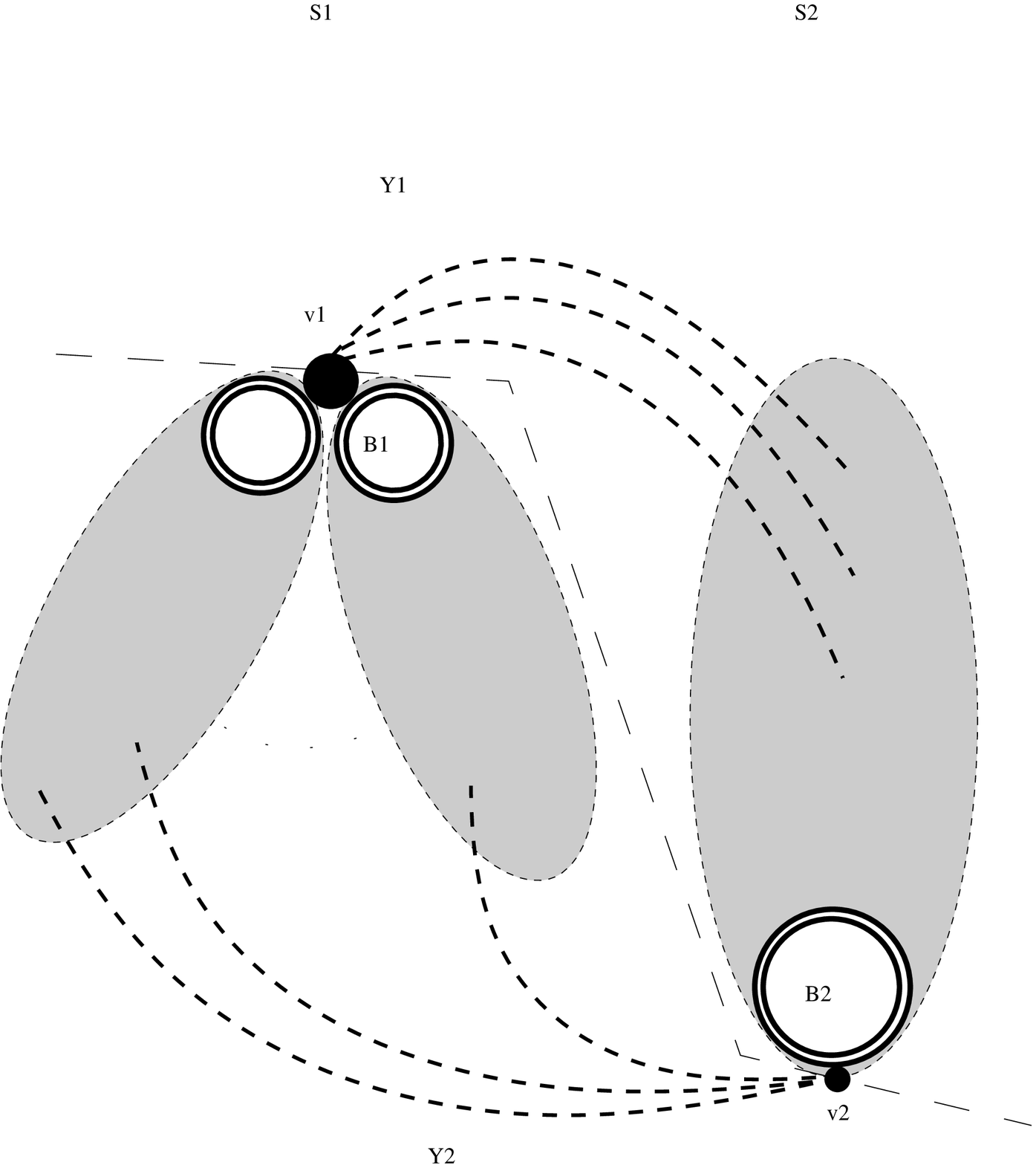}
}\quad
\subfigure[]
{
\includegraphics*[scale=0.265]{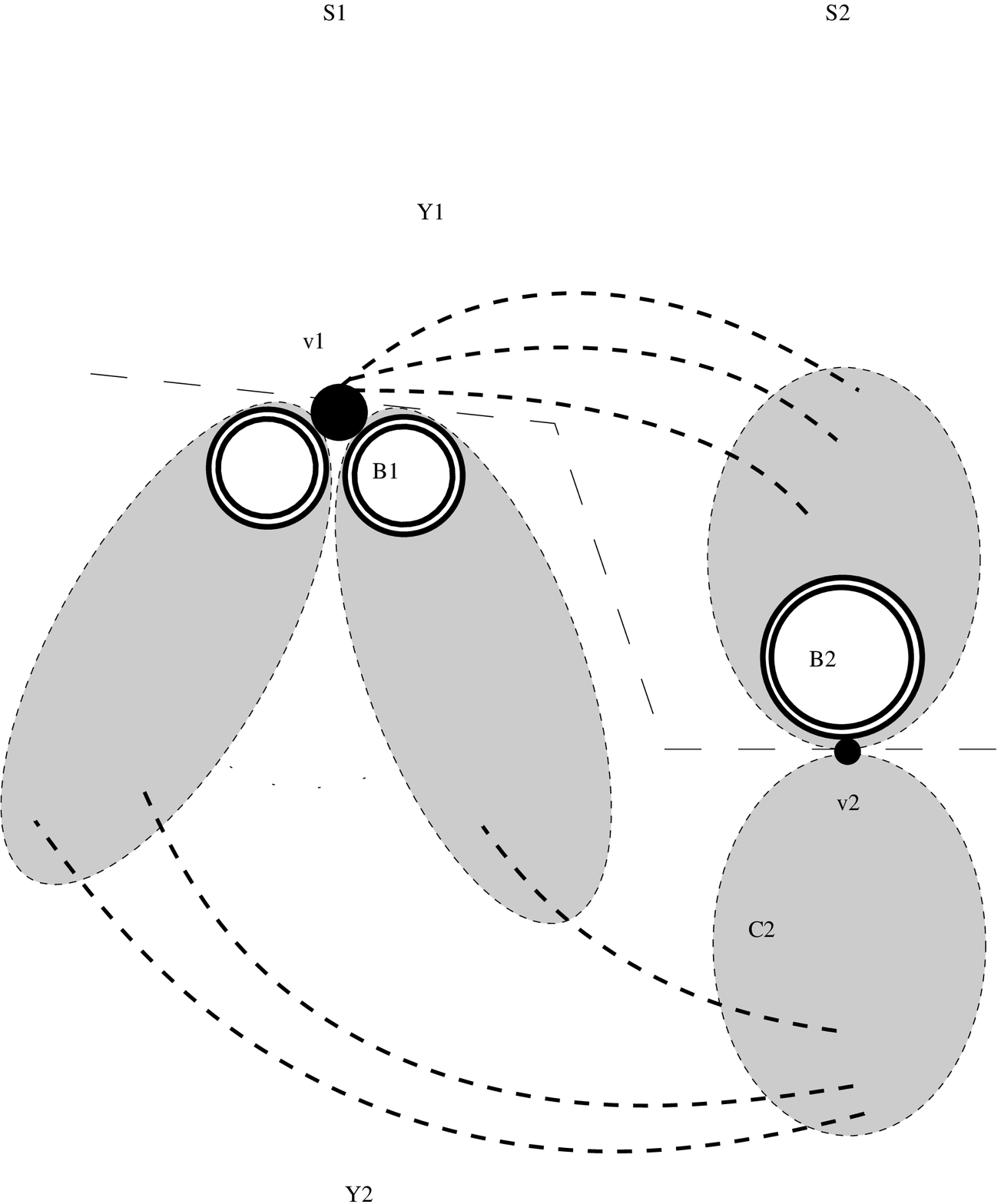}
}\quad
\subfigure[]
{
\includegraphics*[scale=0.265]{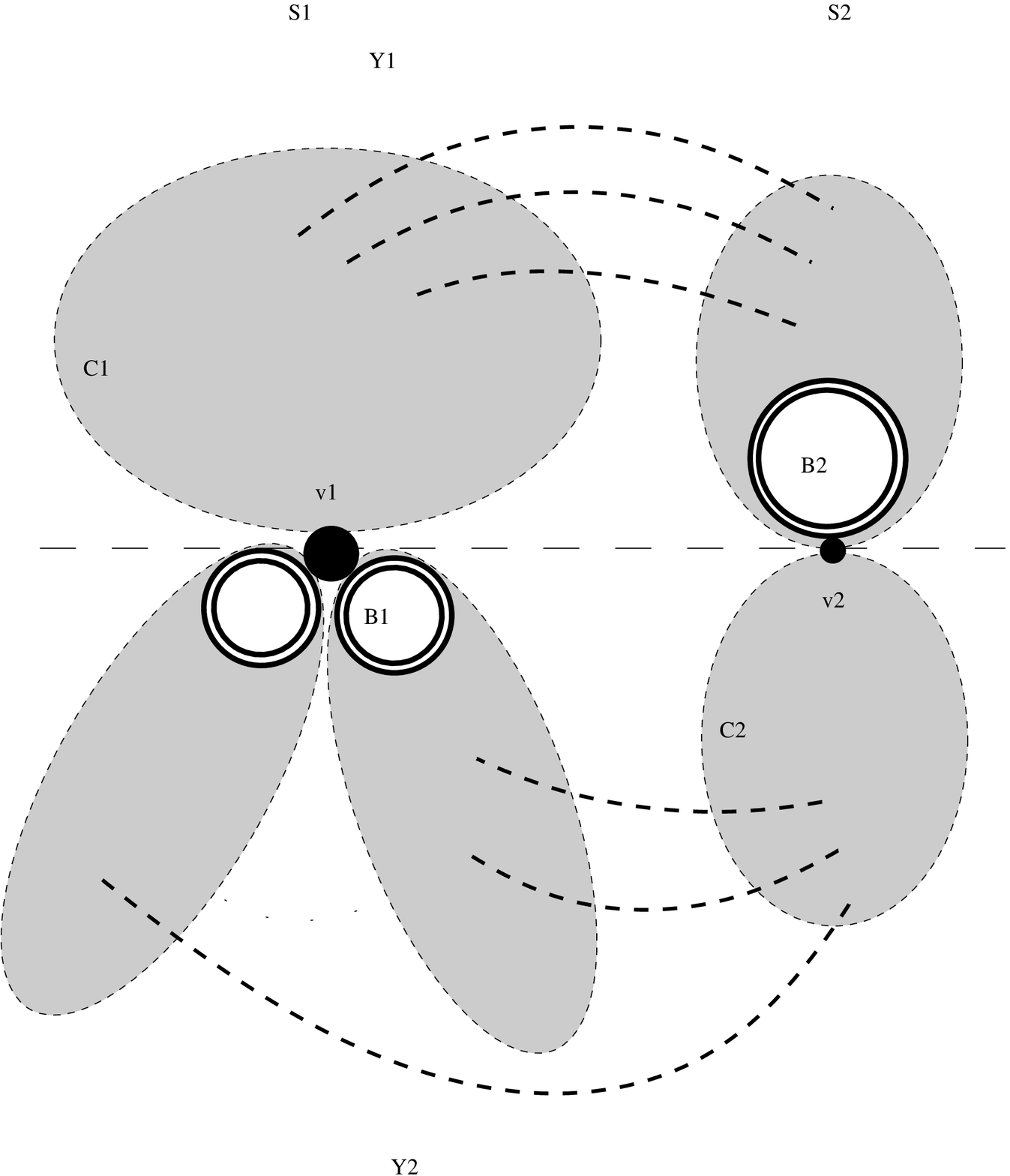}
}} \caption{2-separations from non-overlapping bridges}
\label{fig_2-separations}
\end{center}
\end{figure}

Consider the first case. The reversing parts of $\Sigma$ about $v_{1}$ and $v$ are
$\Sigma_{1}' = \Sigma : E_{1}$ and $\Sigma_{2}' = \Sigma : (E(\Sigma)-E_{1})$ where $E_{1} = E(\Sigma_{1}) \cup Y(B_{0}, v)$. We first assume that 
$B^{-} \subseteq E(\Sigma_{2})$.  If $\Sigma_{1}'$ contains no negative cycles, then by Proposition~\ref{prop_samematroid}  $M(\Sigma') = M(\Sigma)$.  
Let $C_{1}$ be a negative cycle in $\Sigma_{1}'$. Since the only negative edges of $\Sigma_{1}'$  appear in $Y(B_{0},v)$ we will have that  
$E(C_{1}) \cap Y(B_{0}, v) \neq \emptyset$ which in turn implies that $v \in V(C_{1})$. Moreover, $B_{0} = B^{-}$ because otherwise, $C_{1}$ would be vertex disjoint 
with all the negative cycles contained in $B^{-}$, a contradiction since $\Sigma$ is tangled.  Now, $\Sigma_{2}'$ will contain a negative cycle $C_{2}$ such that 
$v \notin V(C_{2})$, otherwise $v$ would be a balancing vertex in $\Sigma$, contradicting the fact that $M(\Sigma)$ is a  non-graphic matroid (see 
Proposition~\ref{prop_sgn_graphic}).  Therefore $C_{2}$ is not contained in $B^{-}$, which implies that  
$E(C_{2}) \cap \bigcap_{B_{i}\in Q(v)}Y(B_{i}, v_{1}) \neq \emptyset$. In that case $v_{1} \in V(C_{1})$, otherwise $C_{1}$ and $C_{2}$ would be vertex disjoint. 
We can therefore conclude that any negative cycle in $\Sigma_{1}'$ contains the vertices $v_{1}$ and $v$, and by Proposition~\ref{prop_samematroid} 
$M(\Sigma') = M(\Sigma)$. If now we assume that $B^{-} \subseteq E(\Sigma_{2})$ then the proof follows the above lines with the minor difference that $B^{-}$ can 
be any $B_{i} \in Q(v)$.

Consider now that $C(B_{0},v)= \emptyset$ and $\bigcap_{B_{i}\in Q(v)} C(B_{i},v_{1}) \neq \emptyset$. The reversing parts of $\Sigma$ about $v_{1}$ and $v$ are 
$\Sigma_{1}' = \Sigma : E_{1}$ and $\Sigma_{2}' = \Sigma : (E(\Sigma)-E_{1})$ where $E_{1} = E(\Sigma_{1}) \cup Y(B_{0}, v) \cup E(C(B_{0}, v))$. We shall consider 
cases regarding the subgraph containing $B^{-}$. In all cases we assume that both reversing parts are unbalanced, otherwise   $M(\Sigma') = M(\Sigma)$ due to 
Proposition~\ref{prop_samematroid} and thus, there is nothing to prove.  Assume that $B^{-}$ is in $\Sigma_{1}'$, and consider any negative cycle $C_{1}$ in $\Sigma_{2}'$. 
If $B^{-} \subseteq E(C(B_{0}, v))$, then  since $E(C_{1}) \cap \bigcap_{B_{i}\in Q(v)}Y(B_{i}, v_{1}) \neq \emptyset$ and 
$v_{1} \in V(C_{1})$, for $C_{1}$ not to be vertex disjoint with some 
negative cycle in $B^{-}$ we must have that $v \in V(C_{1})$ also. If $B^{-} \subseteq E(\Sigma_{1})$, then  $B^{-}$ is any $B_i\in Q(v)$, and we have a case similar 
 to (a).  Alternatively, let $B^{-}$ be in $\Sigma_{2}'$, and let $C_{1}$ be a negative cycle in $\Sigma_{1}'$. Then 
$E(C_{1}) \cap Y(B_{0}, v) \neq \emptyset$ and 
$v \in V(C_{1})$, while $B_{0} = B^{-}$. In order for $v$ not to be a balancing vertex there must exist a cycle $C_2$ in $\Sigma_2'$ which will have $v$ as a vertex. 
However, since $\Sigma$ is tangled, $C_1$ and $C_2$ must not be vertex disjoint and  therefore, $v_{1} \in V(C_{1})$. Thus, in all cases $C_{1}$ 
contains both vertices $v_1$ and $v$ and by Proposition~\ref{prop_samematroid} $M(\Sigma') = M(\Sigma)$. The case (c) is similar to the previous cases. 
\end{proof}

We are now in a position to prove the main result of this paper. 
\begin{theorem}\label{thrm_decomposition}
Let $M$ be a connected binary matroid and $Y\in \mathcal{C}^{*}(M)$ be a non-graphic cocircuit. Then $M$ is signed-graphic if and only if:
\begin{itemize}
 \item [(i)] $Y$ is bridge-separable, and
 \item [(ii)] the $Y$-components of $M$ are all graphic apart from one which is signed-graphic.
\end{itemize}
\end{theorem}
\begin{proof}
Assume that $M$ is signed-graphic. Since it is  binary and not graphic, by Theorem~\ref{th_sliii} there exists a tangled
signed graph $\Sigma$ such that $M=M(\Sigma)$. Moreover since $M\dof Y$ is not graphic, $Y$ cannot be a balancing bond
or unbalancing bond of $\Sigma$ such that $\Sigma \dof Y$ contains a B-necklace. Therefore $Y$ is either a star bond or 
an unbalancing bond such that $\Sigma \dof Y$ does not contain a B-necklace, and by Theorem~\ref{th_louloudi}
we can conclude that $Y$ is a bridge-separable cocircuit of $M$. By Theorem~\ref{thrm_tangled2}, $\Sigma \dof Y$ will
contain exactly one unbalanced block, say  $\Sigma \dto B^{-}$ which is not a B-necklace, and $k$ balanced blocks $\Sigma \dto B_{i}$
where $k\geq 0$. By Theorem~\ref{th_Zasl11} these blocks are the elementary separators of $M(\Sigma\dof Y) =M(\Sigma)\dof Y$, and 
therefore the bridges of $Y$ in $M(\Sigma)$. By Lemma~\ref{lem_balanced=graphic}
we have that $M(\Sigma)\cto (B_{i}\cup Y)$ is graphic for each $i$.  Now, since $M(\Sigma)\cto (B^{-}\cup Y)$ is a minor
of $M(\Sigma)$ it can be either signed-graphic or graphic. It cannot be graphic though since otherwise we would have a 
bridge-separable cocircuit of a connected binary matroid with all $Y$-components graphic, and by Theorem (8.5) of Tutte 
in~\cite{Tutte:1959}, $M$ should be graphic.

Assume by  contradiction that there exists $M$ and $Y\in  \mathcal{C}^{*}(M)$ such that the theorem is not true, and
among these choose the one with the least $| E(M) |$.

If $Y$ has only one bridge $B$, then $M=M\cto ( B\cup Y)$ and $M$ is signed-graphic by assumption.
Given that $Y$ has more than one bridge and its bridge-separable, we partition the bridges of $Y$ into two nonempty families
$L^{-}$ and $L^{+}$ such that no two members of the same family overlap. Furthermore let $B^{-}\in L^{-}$, where $B^{-}$ is the 
bridge of $Y$ corresponding to the unique signed-graphic component $M\cto ( B^{-} \cup Y^{-})$. Let $U^{-},U^{+} \subseteq E(M)$ 
be the unions of the members of $L^{-}$ and $L^{+}$ respectively. 

Let us now consider the matroids $M\cto (U^{-}\cup Y)$ and $M\cto (U^{+} \cup Y)$, and let $U$ denote both $U^{+}$ and $U^{-}$. 
By Theorem (2)
in~\cite{Tutte:1960} we know that $M\cto (B \cup Y)$ is connected for any $B\in U$. If $S$ is a separator of $M\cto (U \cup Y)$, then 
there exists some $B\in U$, such that $S \cap (B\cup Y) \neq \emptyset$. By the definition of contraction operation then, 
$S\cap (B\cup Y)$ would be a separator in $M\cto (B \cup Y)$. We can therefore conclude that the matroids $M\cto (U^{-}\cup Y)$ 
and $M\cto (U^{+} \cup Y)$ are connected.  Now since $M\cto (U\cup Y) = (M ^{*} \dto (U \cup Y))^{*}$ by the definition of 
contraction, we have
\[
(M\cto (U\cup Y) )^{*} = M^{*} \dto (U \cup Y),
\]
which implies that $Y$ is a cocircuit of $M\cto (U\cup Y)$.
By Theorem (7.4) in~\cite{Tutte:1959} we know that the bridges of $Y$ in $M\cto (U\cup Y)$ are the members of $L^{-}$ and $L^{+}$ 
respectively, and $\pi ( M\cto (U\cup Y), B, Y) = \pi ( M, B,  Y) $ for all $B\in U$, which means that the bridges of $Y$ are 
non-overlapping in both matroids. Moreover the $Y$-components in both $M\cto (U^{-} \cup Y)$ and  $M\cto (U^{+} \cup Y)$ are 
the $Y$-components in $M$, since $M\cto (U\cup Y) \cto (B \cup Y) = M \cto (B \cup Y)$.
We can therefore conclude that $M\cto (U^{+} \cup Y)$ is a graphic matroid by Theorem (8.5) of~\cite{Tutte:1959}, while
$M\cto (U^{-} \cup Y)$ is signed-graphic since its smaller than $M$, and $Y$ is a cocircuit with non-overlapping bridges in both. 

By Theorem (8.4) in~\cite{Tutte:1959} there exists a 2-connected graph $G^{+}$ such that $M\cto (U^{+} \cup Y) = M(G^{+})$ and $Y$ 
will be a star at a vertex say $w^{+}$. By Theorem~\ref{th_tu84} there exists a 2-connected tangled signed graph 
$\Sigma^{-} := (G^{-}, \sigma^{-} )$ such that $M\cto (U^{-} \cup Y ) = M(\Sigma^{-})$ and $Y$ is a star bond say at vertex $v^{-}$. 
Construct now a signed graph $\Sigma := (G, \sigma)$ as follows. The underlying graph $G$ is obtained by the graphs $G^{+}\dof Y$ 
and $G^{-}\dof Y$ were $w^{-}$ and $w^{+}$ are deleted, by adding an edge between every end-vertex of $Y$ in $G^{+}$ to the 
corresponding end-vertex of this edge in $G^{-}$. The sign function $\sigma$ will be
\begin{equation*}
\sigma (e) := \left\{ \begin{array}{ll}
                    \sigma^{-} (e),   & \mbox{if  $e\in E(\Sigma^{-})$,} \\
                    +1,   & \mbox{otherwise}.
                \end{array}
          \right. 
\end{equation*}
Since $G^{+}$ and $G^{-}$ are 2-connected and $Y$ is a star of  a vertex in both, $Y$ would be a minimal set of 
edges in $\Sigma$ such that its deletion creates two components, and namely $\Sigma : U^{+}$ and $\Sigma : U^{-}$.
The component $\Sigma : U^{+}$ contains only positive edges by construction,  therefore its balanced. If $\Sigma : U^{-}$ did 
not contain a negative cycle it would imply that $w^{-}$ is a balancing vertex in $\Sigma^{-}$ which contradicts the fact that 
its tangled. We therefore conclude that $Y$ is an unbalancing bond in $\Sigma$ .

Since $\Sigma : U^{+}$ is balanced, we have 
\begin{equation} \label{eq_dec_1}
M(\Sigma) \cto (U^{-} \cup Y) \!=\! M(\Sigma \cto (U^{-} \cup Y) ) \!=\! M(\Sigma\cof U^{+}) \!=\! M(\Sigma^{-}) \!=\! M\cto (U^{-} \cup Y).
\end{equation}
$\Sigma : U^{-}$ contains at least one negative cycle, therefore $\Sigma\cof U^{-}$ is a signed graph with all positive edges, and 
half edges in the end-vertices of $Y$. By Proposition~\ref{prop_sgn_graphic} then
\begin{equation} \label{eq_dec_2}
M(\Sigma) \cto (U^{+} \cup Y) \!=\! M(\Sigma \cto (U^{+} \cup Y) ) \!=\! M(\Sigma\cof U^{-}) \!=\! M(G^{+}) \!=\! M\cto (U^{+} \cup Y).
\end{equation}
Finally, given that both $M\cof (U^{+} \cup Y)$ and $M\cof (U^{-} \cup Y)$ are connected, by a similar argument previously in the 
proof, $M(\Sigma)$ is also connected.

Now that we have established a relationship between $M$ and $M(\Sigma)$ given by (\ref{eq_dec_1}) and (\ref{eq_dec_2}),  by using a
matroidal argument we will show that they are in fact equal.
Consider the following family of cocircuits of $M$
\begin{equation*}
\mathbb{C} := \{ X\in\mathcal{C}^{*}(M(\Sigma)) : \exists \; X_{i} \in \mathcal{C}^{*}(M) \; \mbox{such that} \; X = \triangle X_{i} \}
\end{equation*}
where $\triangle$ stands for the symmetric difference of sets. Note that for $X_{1}, X_{2} \in \mathbb{C}$ such that 
$X_{1} \cap X_{2} \neq\emptyset$, since $M(\Sigma)$ is binary we have that $X_{1} \triangle X_{2} \in \mathcal{C}^{*}(M(\Sigma))$ 
and $X_{1} \triangle X_{2} \in \mathbb{C}$. 
\begin{claim} There exists some $X\in \mathcal{C}^{*}(M(\Sigma)) - \mathbb{C}$ such that $X-(X\cap Y)$ is a 
cocircuit of $M(\Sigma)\dof Y$.
\end{claim}
\begin{proof} 
We can assume that for all $X\in \mathcal{C}^{*}(M(\Sigma)) - \mathbb{C}$ we  have $X\cap Y = \emptyset$, since otherwise
by the deletion operation we have that $X-(X\cap Y) \in \mathcal{C}^{*} ( M(\Sigma)\dof Y)$. Choose such an $X$ and assume that
it is not a cocircuit of $M(\Sigma)\dof Y$. Then there exists $T\in \mathcal{C}^{*} ( M(\Sigma) )$ such that $T \subset X\cup Y$, and
since $M(\Sigma)$ is binary $X\triangle T$ is a cocircuit of $M(\Sigma)$. If $X\triangle T$ or $T$ do not belong to $\mathbb{C}$
the result follows. Therefore $(X\triangle T) \triangle T = X$ which implies that $X \in \mathbb{C}$, which is a contradiction.
\renewcommand{\qedsymbol}{} 
\end{proof}
By the above claim and the fact that $U^{-}$ and $U^{+}$ are separators for $M(\Sigma)\dof Y$ by construction, we can conclude that
either $X\subseteq (U^{+}\cup Y)$ or $X\subseteq (U^{-}\cup Y)$.  Therefore since $X\in \mathcal{C}^{*}(M(\Sigma))$ we have that $X$ is either a cocircuit 
of $M(\Sigma) \cto (U^{+}\cup Y) = M\cto (U^{+} \cup Y)$ or a cocircuit of $M(\Sigma) \cto (U^{-}\cup Y) = M\cto (U^{-} \cup Y)$, and a cocircuit of $M$. 
But since $M$ is connected and binary this is 
a contradiction to the fact that $X\notin \mathbb{C}$. So for any cocircuit $X\in \mathcal{C}^{*}(M(\Sigma))$ we have
\[
X = X_{1} \triangle X_{2} \triangle \ldots \triangle X_{n}
\]
for $X_{i} \in \mathcal{C}^{*}(M)$. But in binary matroids the symmetric difference of cocircuits contains a cocircuit or its empty, 
so we can conclude  that there exists some $X' \in \mathcal{C}^{*}(M)$ such that $X' \subseteq X$.

Reversing the above argument we can also state that any cocircuit $X'$ of $M$ contains a cocircuit $X$ of $M(\Sigma)$, 
and by Lemma (2.1.19) in~\cite{Oxley:92} we have that $M = M(\Sigma)$ contradicting our original hypothesis.
\end{proof}

In Theorem~\ref{thrm_decomposition} a binary signed-graphic $M$ is decomposed given that  it contains  non-graphic separating cocircuits. 
If all the cocircuits of $M$ are graphic, then we can apply the excluded minor characterization given in Theorem~\ref{th_ll3}. It remains to be
shown the case of those binary signed-graphic matroids where all non-graphic cocircuits are non-separating. The following result is a 
chain like  theorem which demonstrates that the non-existence of non-graphic and non-separating cocircuits is preserved by the operation
of deletion of a cocircuit. 

\begin{theorem}\label{thrm_llla}
If $M(\Sigma)$ is a binary signed-graphic matroid such that any non-graphic cocircuit $Y\in{\mathcal{C}^{*}(M(\Sigma))}$ is non-separating then
 any non-graphic cocircuit $Y'$ of $M(\Sigma)\backslash{Y}$ is also non-separating.
\end{theorem}
\begin{proof}
By way of contradiction, suppose that $Y'$ is a non-graphic cocircuit of $M(\Sigma)\backslash{Y}=M(\Sigma\backslash{Y})$ which is separating. 
Let $\Sigma'=\Sigma\backslash{Y}$, then, due to the classification of bonds, $Y'$ is an unbalancing bond of $\Sigma'$. Thereby, 
$\Sigma'\backslash{Y'}$ consists of two components $\Sigma_1'$ and $\Sigma_2'$ which are non-empty of edges, where w.l.o.g. suppose that 
$M(\Sigma_1')$ is not graphic. Moreover, either $Y'$ or $\bar{Y}=Y'\cup{S}$, where $S\subset{Y}$, is a cocircuit of $M(\Sigma)$. If $Y'$ was 
a cocircuit of $M(\Sigma)$ then it could not be a star of $\Sigma$ since $Y'$ is not a star in $\Sigma'=\Sigma\backslash{Y}$. Furthermore, since 
graphicity of a matroid is a minor-closed property, $M(\Sigma)\backslash{Y'}$ is not graphic and therefore, $Y'$ is an unbalancing bond of $\Sigma$ 
which implies that $Y'$ is a non-graphic and separating cocircuit of $M(\Sigma)$; a contradiction. 
In the other case, due to the fact that $Y$ is a star bond of $\Sigma$, all the edges in $S$ have a common end-vertex $v$ in $\Sigma$. Let 
$S_1\subset{S}$ be the edges which have their end-vertices other than $v$ in $\Sigma_1$. Then, $\hat{Y}=Y'\cup{S_1}$ is a minimal set of edges 
such that $\Sigma\backslash{\hat{Y}}$ consists of two components, one of which is $\Sigma_1$, where $M(\Sigma_1)$ is not graphic. This implies that 
$\hat{Y}$ is an unbalancing bond of $\Sigma$ and therefore, $\hat{Y}$ is a separating and non-graphic cocircuit of $M(\Sigma)$; a contradiction.      
\end{proof}

Finally, we provide a graphical characterization for those binary signed-graphic matroids where all non-graphic cocircuits are non-separating. 
\begin{theorem}\label{thrm_nomore}
$M(\Sigma)$ is a connected binary signed-graphic matroid such that any non-graphic cocircuit $Y\in{\mathcal{C}^{*}(M(\Sigma))}$ is non-separating 
if and only if $\Sigma$ is a $2$-connected signed graph such that any non-graphic cocircuit $Y\in \mathcal{C}^{*}(M(\Sigma))$ is a star of $\Sigma$ 
and $\Sigma\backslash{Y}$ is $2$-connected.
\end{theorem}
\begin{proof}
($\Rightarrow$) Since $M(\Sigma)$ is binary and connected, by Theorem~\ref{th_nero1}, $\Sigma$ is tangled and $2$-connected. Furthermore, 
any non-graphic cocircuit of $M(\Sigma)$ is non-separating and therefore, by our classification of bonds, $Y$ is a star bond of $\Sigma$. By 
Theorem~\ref{thrm_tangled2}, $\Sigma\dof{Y}$ must have one unbalanced block. Moreover, $M(\Sigma)\dof{Y}$ is connected since $Y$ is 
non-separating and therefore, by Theorem~\ref{th_Zasl11}, $\Sigma\dof{Y}$ can not be a necklace or contain any other block except for the 
unbalanced one.\\
($\Leftarrow$) $M(\Sigma)$ is non-graphic therefore, by Theorem~\ref{th_sliii}, $\Sigma$  tangled. Furthermore, $\Sigma$ is $2$-connected thus, 
by Theorem~\ref{th_nero1}, $M(\Sigma)$ is connected. Any non-graphic cocircuit $Y$ is such that $\Sigma\dof{Y}$ is $2$-connected; moreover, 
$\Sigma\dof{Y}$ is not a necklace since $Y$ is non-graphic. Thus, by Theorem~\ref{th_Zasl11}, $M(\Sigma)\dof{Y}$ is connected and therefore, 
any non-graphic cocircuit of $M(\Sigma)$ is non-separating.
\end{proof}

We are now ready to see how a binary signed-graphic matroid can be decomposed to graphic matroids and possibly one binary matroid with
no $M^{*}(G_{17})$, $M^{*}(G_{19})$, $F_7$, $F_7^{*}$ minors by the successive deletion of a cocircuit. 
While there exist non-graphic separating cocircuits we simply apply Theorem~\ref{thrm_decomposition}, which dictates that the deletion of
such a cocircuit will result in graphic matroids and one signed-graphic matroid $M(\Sigma)$. If all the non-graphic cocircuits of $M(\Sigma)$
are non-separating then by Theorems~\ref{thrm_llla} and~\ref{thrm_nomore}, it is evident that all these cocircuits will correspond to 
stars in $\Sigma$ and they can be deleted, resulting to either a graphic matroid or a signed-graphic matroid with 
no $M^{*}(G_{17})$, $M^{*}(G_{19})$, $F_7$, $F_7^{*}$ minors.

\bibliographystyle{plain}

\begin{thebibliography}{10}

\bibitem{Arch:81}
D.~Archdeacon.
\newblock A {K}uratowski theorem for the projective plane.
\newblock {\em Journal of Graph Theory}, 5(3):243--246, 1981.

\bibitem{BixCunn:79}
R.E. Bixby and W.H. Cunningham.
\newblock Matroids, graphs, and 3-connectivity.
\newblock In J.A. Bondy and U.S.R. Murty, editors, {\em Graph Theory And
  Related Topics}, pages 91--103. Academic Press, 1979.

\bibitem{Diestel:05}
R.~Diestel.
\newblock {\em Graph Theory}.
\newblock Graduate Texts in Mathematics. Springer, 2005.

\bibitem{Gerards:90}
A.M.H. Gerards.
\newblock {\em Graphs and polyhedra. Binary spaces and cutting planes}.
\newblock CWI Tract vol. 73. Centrum voor Wiskunde en Informatica, Amsterdam,
  1990.

\bibitem{Hlileny:07}
P.~Hlileny.
\newblock {MACEK} 1.2+ {MA}troids {C}omputed {E}fficiently {K}it, 2007.
\newblock http://www.fi.muni.cz/$\sim$hlineny/MACEK/.

\bibitem{MoTh:01}
B.~Mohar and C.~Thomassen.
\newblock {\em Graphs on Surfaces}.
\newblock The John Hopkins University Press, 2001.

\bibitem{Oxley:92}
J.G. Oxley.
\newblock {\em Matroid Theory}.
\newblock Oxford University Press, 1992.

\bibitem{Pagano:1998}
S.R. Pagano.
\newblock {\em Separability and representability of bias matroids of signed
  graphs}.
\newblock PhD thesis, Binghampton University, 1998.

\bibitem{Seymour:1980}
P.D. Seymour.
\newblock Decomposition of regular matroids.
\newblock {\em Journal of Combinatorial Theory Series B}, 28:305--359, 1980.

\bibitem{Slilaty:2005b}
D.~Slilaty.
\newblock On cographic matroids and signed-graphic matroids.
\newblock {\em Discrete Mathematics}, 301:207--217, 2005.

\bibitem{Slilaty:06}
D.~Slilaty.
\newblock Bias matroids with unique graphical representations.
\newblock {\em Discrete Mathematics}, 306:1253--1256, 2006.

\bibitem{Slilaty:07}
D.~Slilaty.
\newblock Projective-planar signed graphs and tangled signed graphs.
\newblock {\em Journal of Combinatorial Theory Series B}, 97:693--717, 2007.

\bibitem{SliQin:07}
D.~Slilaty and H.~Qin.
\newblock Decompositions of signed-graphic matroids.
\newblock {\em Discrete Mathematics}, 307:2187--2199, 2007.

\bibitem{SliQinZh:09}
H.~Qin and D.~Slilaty and X.~Zhou.
\newblock The regular excluded minors for signed-graphic matroids.
\newblock {\em Combinatorics, Probability and Computing}, 18:953--978, 2009.

\bibitem{Truemper:98}
K.~Truemper.
\newblock {\em Matroid Decomposition}.
\newblock Leibniz, 1998.

\bibitem{Tutte:1959}
W.T. Tutte.
\newblock Matroids and graphs.
\newblock {\em Transactions of the American Mathematical Society}, 90:527--552,
  1959.

\bibitem{Tutte:1960}
W.T. Tutte.
\newblock An algorithm for determining whether a given binary matroid is
  graphic.
\newblock {\em Proceedings of the American Mathematical Society}, 11:905--917,
  1960.

\bibitem{Tutte:1965}
W.T. Tutte.
\newblock Lectures on matroids.
\newblock {\em Journal of Research of the National Bureau of Standards (B)},
  69:1--47, 1965.

\bibitem{Tutte:98}
W.T. Tutte.
\newblock {\em Graph Theory As I Have Known It.}
\newblock Oxford University Press, 1998.

\bibitem{Tutte:01}
W.T. Tutte.
\newblock {\em Graph Theory}.
\newblock Cambridge University Press, 2001.

\bibitem{Whittle:05}
G.~Whittle.
\newblock Recent work in matroid representation theory.
\newblock {\em Discrete Mathematics}, 302:285--296, 2005.

\bibitem{Zaslavsky:1982}
T.~Zaslavsky.
\newblock Signed graphs.
\newblock {\em Discrete Applied Mathematics}, 4:47--74, 1982; 
\newblock Erratum: {\em Discrete Applied Mathematics},5:248, 1983.

\bibitem{Zaslavsky:1990}
T.~Zaslavsky.
\newblock Biased graphs whose matroids are special binary matroids.
\newblock {\em Graphs and Combinatorics}, 6:77--93, 1990.

\bibitem{Zaslavsky:1991a}
T.~Zaslavsky.
\newblock Biased graphs. {II}. {T}he three matroids.
\newblock {\em Journal of Combinatorial Theory Series B}, 51:46--72, 1991.

\end{thebibliography}

\end{document}